\documentclass[11pt]{amsart}
\usepackage[T1]{fontenc}
\usepackage[latin1]{inputenc}
\usepackage{typearea}
\usepackage{ulem}
\usepackage{amsmath}
\usepackage{amssymb}
\usepackage{latexsym}
\usepackage{bbm}
\usepackage{enumerate}
\usepackage{amsthm}
\usepackage[all]{xy}
\usepackage{hhline}
\usepackage{epsf} %fuer bildchen
\usepackage{cite}
\usepackage{verbatim}
\usepackage{mathtools}
\usepackage{verbatim}
\usepackage{dsfont}
\usepackage[top=2.5cm,left=2cm, right=2cm, bottom=2cm]{geometry}
\renewcommand{\1}{\mathds{1}}
\usepackage{color}
\newcommand{\red}{\color{black}}

\newcommand{\blue}{\color{black}}

\newtheorem{theorem}{{\sc Theorem}}[section]
\newtheorem{corollary}[theorem]{{\sc Corollary}}

\newtheorem{lemma}[theorem]{{\sc Lemma}}
\newtheorem{proposition}[theorem]{{\sc Proposition}}

\theoremstyle{remark}
\newtheorem{remark}[theorem]{{\sc Remark}}

\theoremstyle{definition}

\newtheorem{example}[theorem]{\sc example}
\newtheorem{condition}[theorem]{\sc condition}

\newcommand{\R}{\mathbb{R} }
\newcommand{\N}{\mathbb{N} }

\newcommand{\B}{\mathcal{B}}
\newcommand{\F}{\mathcal{F}}

\newcommand{\D}{\mathcal{D}}

\newcommand{\Prob}{\mathbb{P}}
\newcommand{\E}{\mathbb{E}}

\providecommand{\abs}[1]{\lvert #1\rvert}
\providecommand{\babs}[1]{\bigl\lvert #1\bigr\rvert}
\providecommand{\Babs}[1]{\Bigl\lvert #1\Bigr\rvert}

\providecommand{\norm}[1]{\lVert #1\rVert}

\DeclareMathOperator{\Var}{Var}

\DeclareMathOperator{\Cov}{Cov}

\DeclareMathOperator{\Inf}{Inf}

\DeclarePairedDelimiter{\floor}{\lfloor}{\rfloor}

\renewcommand{\phi}{\varphi}
\renewcommand{\epsilon}{\varepsilon}

\renewcommand{\rho}{\varrho}
\renewcommand{\P}{\mathbb{P}}
\allowdisplaybreaks

\begin{document}

\title[Functional de Jong theorems]{The multivariate functional de Jong CLT}
%\renewcommand{\thefootnote}{\fnsymbol{footnote}}
%\author{Christian D\"obler \and Miko\l aj Kasprzaki \and Giovanni Peccatii}

\author{{ Christian D\"obler, Miko{\l}aj Kasprzak and Giovanni Peccati}}
\date{ \today}

\begin{abstract} 
We prove a multivariate functional version of de Jong's CLT (1990) yielding that, given a sequence of vectors of Hoeffding-degenerate U-statistics, the corresponding empirical processes on 
$[0,1]$ weakly converge in the Skorohod space as soon as their fourth cumulants in $t=1$ vanish asymptotically and a certain strengthening of the Lindeberg-type condition is verified. As an application, we lift to the functional level the `universality of Wiener chaos' phenomenon first observed in Nourdin, Peccati and Reinert (2010).
 
\smallskip

\smallskip


\end{abstract}
\maketitle

\section{Introduction and (simplified) statements of main results}\label{intro}

\subsection{Motivation}\label{motivation}

{\it Degenerate $U$-statistics} (see Section \ref{Ustats} for definitions) are the {\blue fundamental components} of {\it Hoeffding decompositions} for generic random variables of the type $f(X_1,...,X_n)$, with $X_1,...,X_n$ an independent sample, and are thus pivotal objects in stochastic analysis \cite{H48, Serf80, KB94, ES81, LRR16}. A special case of degenerate $U$-statistic is given by {\it homogeneous sums} \cite{MDO, NPR, Rot1}, which are in turn the {\blue building blocks} of {\it Gaussian and Poisson Wiener chaoses} \cite{PT11, NP12}, as well as the basis of the {\it Fourier-Walsh expansion} of mappings defined on discrete cubes \cite{OD14, GS15}.  For future reference, we recall that a homogeneous sum of order $p\geq 1$ is a random variable $Q_n(f_n,\mathfrak{X})$ with the form 
\[Q_n(f_n,\mathfrak{X})=\sum_{i_1,\dotsc,i_p=1}^n f_n(i_1,\dotsc,i_p)\prod_{l=1}^p X_{i_l}\,,\]
where $\mathfrak{X}=(X_i)_{i\in\N}$ is a sequence of independent, mean zero and unit variance random variables and 
$f_n$ is a symmetric kernel (i.e. $f_n(i_1,\dotsc,i_n)=f_n(i_{\pi(1)},\dotsc i_{\pi(n)})$ for all permutations $\pi$ of $\{1,\dotsc,n\}$ and all $1\leq i_1,\dotsc,i_n\leq n$) satisfying $f_n(i_1,\dotsc,i_p)=0$, whenever there are $j\not=l$ such that $i_l=i_j$. 

\smallskip

In the landmark paper \cite{deJo90}, P. de Jong proved the following surprising result: {\it If, for $n\in\N$, $W_n$ is a (non-symmetric) normalized degenerate $U$-statistic of a fixed order $p$, then $W_n$ converges in distribution to a standard normal random variable $Z$, provided $\lim_n \E[W_n^4] = \E[Z^4] = 3$ and a further Lindeberg-Feller negligibility condition is satisfied. }

\smallskip

De Jong's theorem is the ancestor of the class of `fourth moment theorems' established in the last decade for elements of the Wiener chaos of Gaussian and Poisson random measures -- see e.g. \cite[Chapter 5]{NP12}, as well as \cite{DP18b, DVZ}.

\smallskip

In \cite{DP17}, the authors used Stein's method in order to establish quantitative uni- and multivariate counterparts to de Jong's theorem. The techniques introduced in \cite{DP17} put forward several combinatorial quantities that proved useful in further situations: for instance, as demonstrated by the references \cite{DP18, DKP, Doe20} the technical findings of \cite{DP17} allow one to derive quantitative (functional) CLTs for symmetric $U$-statistics that are expressed in terms of purely analytical quantities, that is, norms of so-called of \textit{contraction kernels}.  

\smallskip

The aim of the present paper is to prove a multivariate functional version -- in the sense of weak convergence in the Skorohod space of c\`adl\`ag mappings -- of the main results of \cite{deJo90, DP17}, under a slightly stronger Lindeberg condition (see Condition \ref{lindcond} below). As a direct application, in Section \ref{homsums} we derive a functional version of the {\it universality of Wiener chaos} phenomenon first detected in \cite{NPR}. 
\smallskip

Our main findings appear in Theorem \ref{maintheo} and Theorem \ref{relcomp} below. To the best of our expertise, these statements are the first examples of fourth moment theorems for random elements taking value in a metric space without a Hilbert space structure. See \cite{BC20} for related results in a Dirichlet/Hilbert setting.

\bigskip

\noindent{\bf Acknowledgments}.
This work is part of the first author's habilitation thesis at Heinrich Heine Universit\"at
D\"usseldorf. The research leading to this paper was supported by the FNR grant \textbf{FoRGES(R-AGR-3376-10)}
at Luxembourg University. This work is also part of project \textbf{Stein-ML} that has received funding from the
European Union's Horizon 2020 research and innovation programme under the Marie Sk\l odowska-Curie
grant agreement \textbf{No 101024264}. We thank an anonymous referee for several useful remarks. {\it Data sharing not applicable to this article as no datasets were generated or analysed during the current study.}

%One of the main interests in studying non-symmetric degenerate $U$-statistics in independent random variables $X_1,\dotsc,X_n$ might be that they are -- via the celebrated \textit{Hoeffding decomposition} (see Section \ref{Ustats} below) -- the building blocks of any functional $f(X_1,\dotsc,X_n)$ of $X_1,\dotsc,X_n$. This observation also highlights the importance of controlling the joint asymptotic distributional behaviour of such $U$-statistics that was investigated in \cite{DP17}.
%Moreover, the class of (non-symmetric) degenerate $U$-statistics of a fixed order $p$ in particular contains the class of so-called \textit{homogeneous sums} of order $p$, which themselves are generalizations of quadratic forms. Recall that a homogeneous sum of order $p$ is a random variable $Q_n(f_n,\mathfrak{X})$ of the form 
%\[Q_n(f_n,\mathfrak{X})=\sum_{i_1,\dotsc,i_p=1}^n f_n(i_1,\dotsc,i_p)\prod_{l=1}^p X_{i_l}\,,\]
%where $\mathfrak{X}=(X_i)_{i\in\N}$ is a sequence of independent, mean zero and unit variance random variables and where 
%$f_n$ is a symmetric function (i.e. $f_n(i_1,\dotsc,i_n)=f_n(i_{\pi(1)},\dotsc i_{\pi(n)})$ for all permutations $\pi$ of $\{1,\dotsc,n\}$ and all $1\leq i_1,\dotsc,i_n\leq n$) that satisfies $f_n(i_1,\dotsc,i_p)=0$, whenever there are $j\not=l$ such that $i_l=i_j$. 

%\subsection{Further links to the literature}

%{\blue For example, we may point out connections to fourth moment theorems on Gaussian and Poisson spaces, the Malliavin-Stein method etc. I have briefly compared the main result to existing relevant FCLTs in Remark \ref{mtrem}. }

\subsection{Setup and main theorems}
We now introduce our objects of study, fix some notation and conditions, and state our main results. For the sake of generality, and at the cost of a slightly heavier notation than the one adopted in \cite{deJo90, deJo89}, from now on we will work in the framework of triangular arrays. 

{\blue Throughout this section}, we fix positive integers $d$ and $1\leq p_1< p_2<\ldots< p_d$, where $d$ always denotes dimension. For $m\in\N$, suppose that $X_1^{(m)},\dotsc,X_{n_m}^{(m)}$ are independent random elements, defined on a suitable probability space $(\Omega_m,\F_m,\Prob_m)$ and assuming values in the respective 
measurable spaces $(E_1^{(m)},\mathcal{E}_1^{(m)}),\dotsc,(E_{n_m}^{(m)},\mathcal{E}_{n_m}^{(m)})$. We will denote by $\E_m$, $\Var_m$ and $\Cov_m$ the expectation, variance and covariance with respect to $\P_m$, respectively.
We will systematically assume that $\lim_{m\to\infty} n_m=+\infty$ and,
to simplify the notation, we will sometimes suppress the dependence on $m$. 
For $k=1,\dotsc,d$ and $m\in\N$ we assume that 
$W(k)=W^{(m)}(k)$ is a \textbf{ (completely) degenerate, non-symmetric $U$-statistic} of order $p_k$ based on $X_1^{(m)},\dotsc,X_{n_m}^{(m)}$, with the form
\[W(k)=W^{(m)}(k)=\sum_{J\in \D_{p_k}(n_m)} W^{(m)}_J(k)=\sum_{J\in \D_{p_k}(n_m)} W_J(k)\,,\]
where $W_J(k)$ is a degenerate kernel depending on $\{X_ n : n\in J\}$ -- see Section \ref{Ustats} below. {\blue Here and in what follows, for $p,n\in\N$, we write 
\begin{equation*}
 \D_p(n):=\{J\subseteq[n]\,:\, |J|=p\}
\end{equation*}
for the collection of all $\binom{n}{p}$ distinct $p$-subsets of $[n] = \{1,...,n\}$.}

W.l.o.g. we will assume that $\Var_m(W^{(m)}(k))=1$ for all $k=1,\dotsc,d$ and all $m\in\N$. 
 By $W=W^{(m)}$ we denote the vector $W=(W(1),\dotsc,W(d))^T$ and, for $t\in[0,1]$ let 
$\mathbf{W}_t=\mathbf{W}^{(m)}_t=(\mathbf{W}^{(m)}_t(1),\dotsc,\mathbf{W}^{(m)}_t(d))^T$ be given by 
\begin{equation}\label{defW1}
\mathbf{W}^{(m)}_t(k):=\sum_{J\in \D_{p_k}(\floor{ n_m t})} W_J^{(m)}(k)=\sum_{1\leq i_1<\ldots<i_{p_k}\leq \floor{ n_m t}} W^{(m)}_{\{i_1,\dotsc,i_{p_k}\}}(k)\,,\quad k=1,\dotsc,d,
\end{equation}
{\blue where  $\floor{x}$ denotes the integer part of the real number $x$.} Then, for each $m\in\N$, the process
\begin{equation}\label{defW2}
\mathbf{W}^{(m)}:=\bigl(\mathbf{W}^{(m)}(1),\dotsc, \mathbf{W}^{(m)}(d)\bigr)^T:= (\mathbf{W}^{(m)}_t)_{0\leq t\leq 1}
\end{equation}
 is an element of the Skorohod space $D([0,1];\R^d)$ of c\`adl\`ag functions from $[0,1]$ to $\R^d$. As anticipated, in this paper we provide sufficient conditions for the sequence $\mathbf{W}^{(m)}$, $m\in\N$, to converge in distribution to a continuous Gaussian process $\mathbf{Z}=(\mathbf{Z}(1),\dotsc,\mathbf{Z}(d))^T$ with respect to the Skorohod topology on $D([0,1];\R^d)$ (see e.g. \cite[Sections 14--15]{Bil1} for details).

Note that, by degeneracy and since the $p_k$'s are pairwise distinct, for $0\leq s\leq t\leq 1$ and $1\leq l,k\leq d$ we have 
\begin{align*}
\Cov\bigl(\mathbf{W}^{(m)}_t(l),\mathbf{W}^{(m)}_t(k)\bigr)&=\E_m\bigl[\mathbf{W}^{(m)}_t(l)\mathbf{W}^{(m)}_t(k)\bigr]
=\delta_{k,l} \sum_{J\in \D_{p_k}(\floor{ n_m s})} \E_m\bigl[W^{(m)}_J(k)^2\bigr]\\
&=\delta_{k,l}\sigma_{m,k}^2\bigl(\floor{ n_m s}\bigr)\,,
\end{align*}
where we set 
\[\sigma_{m,k}^2(j):=\sum_{J\in \D_{p_k}(j)} \E_m\bigl[W^{(m)}_J(k)^2\bigr]\,,\quad k=1,\dotsc,d,\; 0\leq j \leq n_m\,.\]
In particular, we have $\sigma_{m,k}^2(0)=0$ and $\sigma_{m,k}^2(n_m)=\Var(W^{(m)}(k))=1$.

\smallskip

We will now introduce and discuss the Conditions \ref{cond1}, \ref{lindcond} and \ref{fmcond} under which our main findings (Theorems \ref{maintheo} and \ref{relcomp}) are stated. 

\smallskip

\begin{condition}[Convergence of time-changes]\label{cond1}
There are functions $v_k:[0,1]\rightarrow\R$, $1\leq k\leq d$, such that the pointwise limits
\begin{equation}\label{varlim}
\lim_{m\to\infty}\sigma_{m,k}^2\bigl(\floor{ n_m s}\bigr)=v_k(s)\,,\quad s\in[0,1]\,,
\end{equation}
exist.
\end{condition}
%{\blue Maybe, we will need a stronger mode of convergence than pointwise convergence.} 
Note that the $v_k$ are necessarily non-decreasing on $[0,1]$ and satisfy $v_k(0)=0$ as well as $v_k(1)=1$. However, we need not assume any regularity of the functions $v_k$.  

{\blue In order to concretely relate Condition \ref{cond1} to functional results}, we let $\mathbf{Z}=(\mathbf{Z}(1),\dotsc,\mathbf{Z}(d))^T$ be a centered Gaussian process with independent components, defined on a suitable probability space $(\Omega,\F,\P)$ and such that, for $1\leq k\leq d$, the covariance function of $\mathbf{Z}(k)=(\mathbf{Z}_t(k))_{t\in[0,1]}$ is given by 
\begin{equation}\label{covZ}
\Cov\bigl(\mathbf{Z}_s(k),\mathbf{Z}_t(k)\bigr)=\E \bigl[\mathbf{Z}_s(k)\mathbf{Z}_t(k)\bigr]=v_k(s\wedge t)\,,
\end{equation}
implying that $(\mathbf{Z}_1(1),\dotsc,\mathbf{Z}_1(d))^T$ is a standard Gaussian vector in $\R^d$. It is easily verified that the $d$-dimensional process ${\bf Z}$ has the same distribution as
\[
\left(\mathbf{B}_{v_1(t) }(1),..., \mathbf{B}_{v_d(t) }(d) \right)_{ t\in [0,1]},
\]
where $(\mathbf{B}(1),...,\mathbf{B}(d))$ is a standard Brownian motion (initialized at zero) in $\R^d$. This yields in particular that ${\bf Z}$ has a.s.-$\P$ continuous paths if and only if the functions $v_k$'s are continuous. 
\smallskip

Following \cite{deJo89,deJo90,DP17} we further introduce the quantities 
\begin{align*}
\rho_{m,k}^2&:=\max_{1\leq i\leq n_m}\sum_{\substack{J\in\D_{p_k}(n_m):\\ i\in J}} \E_m\bigl[W^{(m)}_J(k)^2\bigr]\quad\text{and}\\
D_{m,k}&:=\sup_{\substack{J\in\D_{p_k}(n_m):\\ \E_m[W^{(m)}_J(k)^2]>0} }\frac{\E_m[W^{(m)}_J(k)^4]}{\E_m[W^{(m)}_J(k)^2]^2}\,.
\end{align*}
Note that, e.g. by Jensen's inequality, one necessarily has that $D_{m,k}\geq1$. 
\begin{condition}[Reinforced Lindeberg condition] \label{lindcond}
 For all $1\leq k\leq d$, $\lim_{m\to\infty}D_{m,k}\rho_{m,k}^2=0$.
\end{condition}
This condition roughly ensures that none of the individual random variables $X_i^{(m)}$, $1\leq i\leq n_m$, has an asymptotically dominant impact on the variance of $W^{(m)}(k)$, $1\leq k\leq d$.  \\

The following condition is inherent to all de Jong type CLTs such as e.g. those proved in \cite{deJo90, DP17}.
 
\begin{condition}[Fourth moment condition]\label{fmcond}
 For all $1\leq k\leq d$ we have $\lim_{m\to\infty}\E_m\bigl[W^{(m)}(k)^4\bigr]=3$.
\end{condition}

The forthcoming Theorems \ref{maintheo} and \ref{relcomp} are the main achievements of this work.

\begin{theorem}\label{maintheo}
Let the above definitions and notation prevail and suppose that Conditions \ref{cond1}- \ref{fmcond} hold. Then, the process $\mathbf{Z}$ has continuous paths and, as $m\to\infty$, the distributions of the sequence $(\mathbf{W}^{(m)})_{m\in\N}$ weakly converge to the law of $\mathbf{Z}$, with respect to both the Skorohod and uniform topologies.
\end{theorem}

An assumption such as Condition \ref{cond1} does not appear in the finite-dimensional settings of \cite{DP17,deJo90}. It cannot be easily dispensed with without affecting the unicity of the limit in the previous statement. The following result shows that, if Condition \ref{cond1} is removed, then the sequence $(\mathbf{W}^{(m)})_{m\in\N}$ is still relatively compact, and each limit law is the law of a continuous Gaussian process. 

\begin{theorem}\label{relcomp}
Suppose that the sequence $(\mathbf{W}^{(m)})_{m\in\N}$ satisfies Conditions \ref{lindcond} and \ref{fmcond}. Then, the collection of their laws is tight, hence relatively compact, in the class of distributions on the Skorohod space equipped with either the Skorohod or the uniform metric. Moreover, if the laws of a subsequence {\blue $(\mathbf{W}^{(m_j)})_{j\in\N}$} converge weakly, then the limit is the law of a continuous, {\red $d$-dimensional} centered Gaussian process {\red with independent components}. 
\end{theorem}

The detailed proofs of Theorems \ref{maintheo} and \ref{relcomp} are deferred to Section \ref{Proofs}. {{} We observe that the proof of Theorem \ref{relcomp} implicitly shows that, if $(\mathbf{W}^{(m)})_{m\in\N}$ satisfies Conditions \ref{lindcond} and \ref{fmcond}, then from every subsequence $(\mathbf{W}^{(m_j)})_{j\in\N}$ one can extract a sub-subsequence $(\mathbf{W}^{(m'_j)})_{j\in\N}$ verifying Condition 1.1 for some continuous functions $v_k$, whose definition depends on the choice of the subsequence. 

\medskip

Some representative examples of kernels directly verifying the assumptions of Theorem \ref{maintheo} are described in Example \ref{ex:fractional}, yielding in particular that, for all $p\geq 3$ and all integers $a\in \{2,...,p-1\}$, there exists a sequence of homogeneous sums whose associated empirical processes converge in distribution to a multiple of $B(t^{p/a})$, where $B$ is a standard Brownian motion issued from zero. It is important to notice that such a limit behaviour is in principle {\it not} achievable in the framework of symmetric and degenerate $U$-statistics. Indeed, in \cite[Corollary 3.7]{DKP} it is proved that, given a sequence of symmetric and degenerate $U$-statistics of order $p\geq 2$ verifying asymptotic relations that are roughly equivalent to Conditon \ref{fmcond}, the corresponding sequence of normalized empirical processes always converges in distribution to a multiple of $B(t^p)$.
}

\begin{remark}\label{mtrem}
\begin{enumerate}[(a)]
 \item Theorem \ref{maintheo} provides a strong functional extension of the finite-dimensional de Jong type theorems from \cite{deJo90,DP17} under very mild additional assumptions. {\blue As already observed, Condition \ref{cond1} and Condition \ref{fmcond} are indeed natural in this context \cite{deJo90,DP17}}, so that only Condition \ref{lindcond} might appear non-optimal, since ---whenever there is a $k$ such that $D_{k,m}$ is unbounded in $m$ --- it is strictly stronger than the usual negligibility condition demanding that $\lim_{m\to\infty}\rho_{m,k}^2=0$ for all $1\leq k\leq d$, see again \cite{deJo90,DP17}. To address this issue, we observe first that, in many instances, the $D_{k,m}$ are in fact bounded in $m$ for all $1\leq k\leq d$ and in this case, the two conditions are in fact equivalent. This happens in the case of homogeneous sums such that the underlying sequence $\mathfrak{X}$ has a uniformly bounded fourth moment (see Section \ref{homsums} below). Moreover, it is not clear whether the convergence of finite dimensional distributions of $\mathbf{W}$ to those of $\mathbf{Z}$ would still hold, if one replaced Condition \ref{lindcond} with the weaker variant (see Subsection \ref{fidi} for details). Some preliminary computations have indeed shown that, if one only assumed that $\lim_{m\to\infty}\rho_{m,k}^2=0$ for all $1\leq k\leq d$, one might have to replace Condition \ref{fmcond} with the assumption that $\lim_{m\to\infty}\E[\mathbf{W}_t(k)^4]=3 v_k(t)^2$ for all $t\in[0,1]$ and all $1\leq k\leq d$, which would be much more complicated to verify than Condition \ref{fmcond}. 
%Finally, as the quantities $D_{k,m}$ also naturally appear in our proof of tightness which usually needs something extra, introducing Condition \ref{lindcond} seems to be a very small sacrifice to us. 
\item Note that Conditions \ref{cond1}- \ref{fmcond}, implying the \textit{joint convergence} of the random processes $\mathbf{W}^{(m)}(k)$, $1\leq k\leq d$, are just the aggregation of the conditions for the \textit{componentwise convergence} of these processes. 
 \item {\blue We find it} remarkable that one uniquely has to take into account the behaviour of the fourth moments of the components of $\mathbf{W}^{(m)}$ at time $1$ and not at other times $t<1$.  
\item We compare our Theorem \ref{maintheo} to a few existing functional CLTs (FCLTs) for degenerate (symmetric and non-symmetric) $U$-statistics in the scarce literature on this topic.
In accordance with the univariate case (see e.g. \cite{DM,Greg77,Serf80}), 
%a symmetric degenerate $U$-statistic of order $p>1$ with a fixed kernel cannot have a Gaussian limit \cite{Neu, H79, MT} but its marginal limiting distributions live in a higher order chaos of a Gaussian process. Hence, 
in order to satisfy a FCLT, a symmetric degenerate $U$-statistic of order $p>1$ must have a variable kernel, i.e. one that depends on the sequential parameter $m$. 
In the recent work \cite{DKP} we have proved FCLTs for this situation with sufficient conditions for convergence that are expressed in norms of contraction kernels and, hence, are completely analytical in nature. 
%We also refer to \cite{DKP} for a more detailed review of the relevant literature on invariance principles for symmetric, degenerate 
%$U$-statistics. 
A more general class of statistics is given by the so-called {\it weighted, degenerate $U$-statistics of order $p$}. These have the form 
$U_n=\sum_{1\leq i_1<\ldots<i_p\leq n}a(i_1,\dotsc,i_p) \psi_n(X_{i_1},\dotsc,X_{i_p})$, where $X_1,\dotsc,X_n$ are i.i.d. elements of some space $E$, $ a(i_1,\dotsc,i_p)$, $1\leq i_1,\dotsc,i_p\leq n$ is a symmetric array of real numbers that vanishes on diagonals (i.e. $  a(i_1,\dotsc,i_p)=0$ if $i_l=i_k$ for some $l\not=k$) and $\psi_n$ is a symmetric, canonical kernel, i.e. 
$\E[\psi_n(X_1,\dotsc,X_p)|X_1,\dotsc,X_{p-1}]=0$ a.s. In particular, they constitute a strict subclass of the non-symmetric $U$-statistics investigated in the present paper. 
%Functional CLTs for weighted, degenerate $U$-statistics are extremely rare in the literature. 
In the case $p=2$, \cite{Mik2} derives functional limit theorems for weighted degenerate $U$-statistics by using 
%the spectral decomposition of the kernel $\psi_n(x,y)$ (see \cite{Greg77, Serf80}) combined with an invariance principle for quadratic forms in independent centered random variables due to \cite{Rot2}. In particular, the method of proof in \cite{Mik2}
methods that heavily depend on the peculiarities of the case $p=2$. In the very recent paper \cite{DK21} the first and the second author have introduced a functional version of Stein's method of exchangeable pairs and applied it to prove quantitative, multidimensional FCLTs for degenerate weighted $U$-statistics. As explained in detail in \cite[Subsection 1.5]{DK21}, the limit theorems that build on the bounds of \cite{DK21} rather involve third absolute moments and absolute values of the coefficient arrays and, hence, are complementary to the present ones involving fourth moment conditions. As to FCLTs for homogeneous sums, apart from our paper \cite{DK21} we only mention the two references \cite{Mik1} and \cite{Basa}, where, again, \cite{Mik1} only considers the case $p=2$ of quadratic forms. The paper \cite{Basa} instead, considers homogeneous sums of arbitrary order $p$ but the results of this paper are in question. Indeed,  we have found that the argument leading to \cite[Theorem 1.1]{Basa} is flawed, as carefully explained in part 2 of \cite[Remark 5.6]{DK21}.
%For instance, on page 187 therein, in the display below (2.9), one cannot just drop the quantity $\tau_n^4$ (not even at the price of a larger absolute constant $C$) because the claimed inequality must hold for all (sufficiently large) fixed values of $n\in\N$ and all $t_1,t_2\in[0,1]$. Moreover, the application of \cite[Theorem 15.6]{Bil1} on page 188 seems to be a bit rushed, since the almost sure left-continuity of the limiting Gaussian process $\xi_k$ at $1$ is not verified. Finally, the claimed limiting process $\xi_k$ appearing in \cite[Theorem 1.1]{Basa} is not even completely determined, since equation (1.4) thereof only specifies the one-dimensional distributions of $\xi_k$ but not its covariance function. 
%Moreover, as for weighted $U$-statistics, the results on FCLTs for homogeneous sums from \cite{DK21} rather involve third moment quantities and absolute values of the coefficient arrays and, hence, are complementary to our fourth moment condition that relies on cancellation effects depending themselves on variable signs.      
\end{enumerate}
\end{remark}

{ \subsection{About our approach}
Our method of proof turns out to be a combination of a quantitative, multivariate CLT for vectors of degenerate $U$-statistics to obtain convergence of finite-dimensional distributions (see Lemma \ref{mdimlemma} below), and a martingale representation, which we exploit via Doob's $L^4$-inequality, in order to derive tightness. In order to conclude from Condition \ref{fmcond} that the fourth cumulant 
of $\mathbf{W}^{(m)}_t(k)$, $1\leq k\leq d$, also converges to zero for all fixed $t\in[0,1)$ as $m\to\infty$, we elaborate on a result from the monograph \cite{deJo90} (see Lemma \ref{cumlemma} below). Moreover, for the application of Lemma \ref{mdimlemma}, we prove and make use of several new results about degenerate $U$-statistics and their Hoeffding decompositions in Section \ref{Ustats}, which are of independent interest.

Since in his original proof of the one-dimensional CLT \cite{deJo89}, de Jong made use of a classical martingale CLT due to Heyde and Brown \cite{HB}, it is a natural question whether one could similarly apply a FCLT for martingales, like for instance \cite[Theorem 2.34]{MPU}, in order to derive Theorem \ref{maintheo}. Indeed, adapting the computations beginning on page 286 of \cite{deJo89} and using our Lemma \ref{cumlemma}, one could verify that (an adaptation to deterministic time changes of) {\blue the ``square bracket condition''} (2.48) in \cite{MPU} holds for all fixed $t\in[0,1]$. Since the {\blue uniform negligibility} condition (2.46) {\blue in \cite{MPU}} is satisfied for any sequence of normalized $L^2$-martingales, \cite[Theorem 2.34]{MPU} could be applied to yield another proof of the \textit{one-dimensional version} of Theorem \ref{maintheo}.
 In order to derive the full multidimensional result, one could try to similarly adapt a multivariate FCLT for martingales to our setting, like for example \cite[Theorem 2.1]{Whitt}. However, the mixed square bracket condition (3) in \cite{Whitt} seems to be less {\blue amenable to analysis in the present setting} than its univariate counterpart, i.e. condition (2.48) in \cite{MPU}. Hence, adapting the techniques used in \cite{deJo90} so as to obtain our multivariate result appears {\blue considerably more difficult}. One of the main reasons for this 
fact is that de Jong's arguments rely {\blue on a large collection of exact combinatorial identities, that one uses in order} to verify the square bracket condition {\blue evoked above}. In contrast, our direct arguments merely use {\blue probabilistic inequalities}, which makes the structure of the proof significantly simpler and probably also more robust.
Finally, we mention that our arguments are in principle quantitative in nature and that, via an elaboration of Stein's method for diffusion approximation \cite{Bar, Kas1, Kas2, DK21}, one  
would be able to prove a quantitative version of Theorem \ref{maintheo}. Since, at the time of writing, we are only in the position to do this with respect to the stronger $L^1$-topology on $D([0,1];\R^d)$, we {\blue decide to leave this point open for further investigation}. 
}

\subsection{Application: universal FCLTS for homogeneous sums}\label{homsums} 

{\blue 
  We will now apply our main results to the setting of homogeneous sums: as a consequence, in Theorem \ref{universal} we will establish functional versions of the main findings of \cite{NPR}, where the authors proved the {\it universality of Gaussian Wiener chaos} with respect to the normal approximation of homogeneous sums. Such a result roughly states that, if a sequence of homogeneous sums with Gaussian arguments verifies a CLT, then so does the sequence obtained by replacing the Gaussian input with an arbitrary collection of independent random variables  (see e.g. \cite[Theorem 1.2]{NPR} for a precise statement). The findings from \cite{NPR} have found applications e.g. in the analysis of disordered systems \cite{CSZ17, CSZ20}, mathematical statistics \cite{koike}, random matrix theory and free probability \cite{NP-alea, NPPS}.
}

\smallskip

In this subsection, we let $(\Omega,\F,\P)$ denote a suitable probability space on which all random quantities subsequently dealt with are defined. We will consider the following sequences of real random variables: 
\begin{itemize}
\item $\mathfrak{X}=(X_i)_{i\in\N}$ denotes a (generic) sequence of independent, mean zero and unit variance random variables in $L^4(\P)$; % such that $\sup_{i\in\N}\E|X_i|^3<\infty$;
\item $\mathfrak{G}=(G_i)_{i\in\N}$ is a sequence of independent standard normal random variables; and
\item $\mathfrak{P}=(P_i)_{i\in\N}$ is a sequence of independent, normalized Poisson random variables, i.e. there are $\lambda_i\in(0,\infty)$ and independent Poisson random variables $N_i$ with mean $\lambda_i$ such that $P_i=N_i/\sqrt{\lambda_i} -\sqrt{\lambda_i}$, $i\in\N$.
\end{itemize}

On the sequence $\mathfrak{X}$, we will further impose the integrability condition that
\begin{equation}\label{ufmb}
\beta:=\sup_{i\in\N}\E|X_i|^4<\infty.
\end{equation}
This is in particular satisfied if the random variables $X_i$, $i\in\N$, are i.i.d, since we have assumed that they are in $L^4(\P)$.
 We observe that, in the case of the Poisson sequence $\mathfrak{P}$, condition \eqref{ufmb} can be equivalently re-expressed in terms  of the coefficients $\lambda_i$. In fact, since $\E[P_i^4]=3 +\lambda_i^{-1}$, the sequence $\mathfrak{P}$ satisfies \eqref{ufmb} if and only if $\inf_{i\in\N}\lambda_i>0$.

In what follows, we let $n_m=m$ and, for given arrays $a^{(m)}_J(k)$, $J\in\D_{p_k}(m)$, $1\leq k\leq d$, of real numbers we suppose that $W_J^{(m)}(k)=a^{(m)}_J(k)\prod_{i\in J}X_i$ is given as 
\[ W_J^{(m)}(k)=a^{(m)}_J(k)\prod_{i\in J}X_i=:a^{(m)}_J(k) X_J\,.\]
Thus, for $k=1,\dotsc,d$,
\begin{equation}\label{defhomsum}
W^{(m)}(k)=\sum_{J\in\D_{p_k}(m)} a^{(m)}_J(k)\prod_{i\in J}X_i
\end{equation}
is a \textit{homogeneous sum} or \textit{homogeneous multilinear form} of order $p_k$ and, in particular, it is a non-symmetric, degenerate $U$-statistic. As before, we will {\blue systematically} assume that 
\[\Var\bigl(W^{(m)}(k)\bigr)=\sum_{J\in\D_{p_k}(m)} a^{(m)}_J(k)^2=1\,,\quad\text{for } k=1,\dotsc,d.\]
Then, as in \eqref{defW1} and \eqref{defW2}, we use these $d$ homogeneous sums to define the random element $\mathbf{W}^{(m)}=(\mathbf{W}^{(m)}(1),\dotsc,\mathbf{W}^{(m)}(d))^T$ with values 
in $D([0,1];\R^d)$.

Note that from \eqref{ufmb}, for $1\leq k\leq d$ and all $J\in\D_{p_k}(m)$ such that $a^{(m)}_J(k)\not=0$, we have that
\[\frac{\E\bigl[ W_J^{(m)}(k)^4\bigr]}{ \Bigr(\E\bigl[ W_J^{(m)}(k)^2\bigr]\Bigr)^2} =\frac{a^{(m)}_J(k)^4 \E\Bigl[\prod_{i\in J}X_i^4\Bigr]}{a^{(m)}_J(k)^4\biggl(\E\Bigl[\prod_{i\in J}X_i^2\Bigr]\biggr)^2}=\prod_{i\in J}\E[X_i^4]\leq\beta^{p_k}\,.\] 
Hence, also
\[\sup_{m\in\N} D_{m,k}\leq\beta^{p_k}<\infty\,.\]

In this context, one would typically like to replace the fourth moment condition, Condition \ref{fmcond}, by a condition, which is rather expressed in terms of the coefficient arrays. To this end, let us first define the symmetric functions $f^{(m)}_k:[m]^{p_k}\rightarrow\R$ by
\[ f^{(m)}_k(i_1,\dotsc,i_{p_k}):=\begin{cases}
\frac{1}{p_k!} a^{(m)}_{\{i_1,\dotsc,i_{p_k}\}}(k)\,,&\text{if } |\{i_1,\dotsc,i_{p_k}\}|=p_k\,,\\
0\,,&\text{otherwise.} 
\end{cases}\]
{{} Observe that the function $f_k^{(m)}$ {\it vanishes on diagonals}, that is, $f^{(m)}_k(i_1,\dotsc,i_{p_k}) = 0$ if two coordinates of $(i_1,\dotsc,i_{p_k})$ coincide.} Then, we have 
\begin{align*}
W^{(m)}(k)&=Q_{p_k}(m,f^{(m)}_k,\mathfrak{X}):=\sum_{i_1,\dotsc,i_{p_k}=1}^m f^{(m)}_k(i_1,\dotsc,i_{p_k}) X_{i_1}\cdot\ldots\cdot X_{i_{p_k}} \\
&=p_k!\sum_{1\leq i_1<\ldots<i_{p_k}\leq m} f^{(m)}_k(i_1,\dotsc,i_{p_k}) X_{i_1}\cdot\ldots\cdot X_{i_{p_k}}
\end{align*}
as well as
\begin{align*}
\rho_{m,k}^2&=\max_{1\leq i\leq m} \sum_{\substack{1\leq j_2<\ldots<j_{p_k}\leq m:\\j_2,\dotsc,j_{p_k}\not=i}} a^{(m)}_{\{i,j_2,\dotsc,j_{p_k}\}}(k)^2
=(p_k!)^2 \Inf_i\bigl(f^{(m)}_k\bigr)  \,,
\end{align*}
where 
\[\Inf_i\bigl(f^{(m)}_k\bigr)=\frac{1}{(p_k-1)!} \sum_{j_2,\dotsc,j_{p_k}=1}^m f^{(m)}_k(i,j_2,\dotsc,j_{p_k})^2=\sum_{1\leq j_2<\ldots<j_{p_k}\leq m}f^{(m)}_k(i,j_2,\dotsc,j_{p_k})^2\]
denotes the \textit{influence} of the random variable $X_i$ on the variance of $W^{(m)}(k)$ (see \cite{MDO, GS15}). Thus, we have shown that, under \eqref{ufmb}, Condition \ref{lindcond} is equivalent to 
\begin{equation}\label{influence}
\lim_{m\to\infty}\max_{1\leq i\leq m}\Inf_i\bigl(f^{(m)}_k\bigr)=0\quad\text{for}\quad 1\leq k\leq d.
\end{equation}

For $1\leq k\leq d$ and $1\leq r\leq p_k-1$, we define the \textit{contraction kernel} $f^{(m)}_k\star_r f^{(m)}_k:[m]^{2p_k-2r}\rightarrow\R$ by 
\begin{align*}
&f^{(m)}_k\star_r f^{(m)}_k(i_1,\dotsc,i_{p_k-r},j_1,\dotsc,j_{p_k-r})\\
&:=\sum_{l_1,\dotsc,l_r=1}^m f^{(m)}_k(i_1,\dotsc,i_{p_k-r},l_1,\dotsc,l_{r})f^{(m)}_k(j_1,\dotsc,j_{p_k-r},l_1,\dotsc,l_{r})\,.
\end{align*}
Consider the following condition.
\begin{condition}\label{contrcond}
For all $1\leq k\leq d$ and all $1\leq r\leq p_k-1$ it holds that $\lim_{m\to\infty}\norm{f^{(m)}_k\star_r f^{(m)}_k}_2=0$.
\end{condition}
Here, $\norm{f^{(m)}_k\star_r f^{(m)}_k}_2$ denotes the $L^2$-norm of the contraction kernel $f^{(m)}_k\star_r f^{(m)}_k$ with respect to the $(2p_k-2r)$-fold product of the counting measure on $[m]$, i.e. with respect to the counting measure on $[m]^{2p_k-2r}$.

\begin{proposition}\label{contrprop}
Suppose that \eqref{ufmb} and Condition \ref{contrcond} hold and that $p_k\geq 2$ for all $1,\dots,d$. Then, Conditions \ref{lindcond} and \ref{fmcond} are satisfied as well.
\end{proposition}

\begin{proof}
We have already shown that Condition \ref{lindcond} holds if and only if \eqref{influence} holds. Relation \eqref{influence} however holds true by virtue of Condition \ref{contrcond} and display (1.9) in \cite{NPR}. That Condition \ref{fmcond} holds under Condition \ref{contrcond} as well, follows from a combination of \cite[Proposition 1.6]{NPR} and the estimate (1.13) in \cite{NPR}.
\end{proof}

Note further that the functions $\sigma_{m,k}^2(\floor{n_m s})$ from Condition \ref{cond1} are in this situation more explicitly given by 
\begin{align}\label{sigmahom}
\sigma_{m,k}^2(\floor{n_m s})&=\sum_{J\in\D_{p_k}(\floor{n_ms})} a_J^{(m)}(k)^2=p_k!\sum_{i_1,\dotsc,i_{p_k}=1}^{\floor{n_ms}} f^{(m)}_k(i_1,\dotsc,i_{p_k})^2=:Sf^{(m)}_k (s)\,.
\end{align}

\begin{theorem}\label{homsumtheo}
Suppose that the sequence $\mathfrak{X}$ of independent, mean zero and unit variance random variables satisfies \eqref{ufmb}. Moreover, assume that the sequences $(Sf^{(m)}_k)_{m\in\N}$, $1\leq k\leq d$, defined in \eqref{sigmahom}, satisfy Condition \ref{cond1}, that Condition \ref{contrcond} holds and that $p_k\geq 2$ for all $k=1,\dots,d$. Then, the sequence $(\mathbf{W}^{(m)})_{m\in\N}$ of processes, defined through \eqref{defW2} and \eqref{defhomsum}, converges in distribution with respect to the Skorohod topology to the continuous, centered Gaussian process $\mathbf{Z}=(\mathbf{Z}(1),\dotsc,\mathbf{Z}(d))^T$ with independent components, which is defined via \eqref{varlim}.
\end{theorem}

\begin{proof}
This follows from Theorem \ref{maintheo} and Proposition \ref{contrprop}.
\end{proof}

The proof of Theorem \ref{homsumtheo} (via the implicit use of Proposition \ref{contrprop}) {\blue strongly relies on the techniques developed in \cite{NPR}. As anticipated, we will now prove a functional version of the universality results from \cite{NPR}.} For completeness, we will also include the discussion of universality of Poisson homogeneous sums, as provided by \cite{PZ, DVZ}. 

\begin{theorem}[Functional universality of homogeneous sums]\label{universal}
With the above notation and definitions, consider the following assumptions: 
\begin{itemize}
\item[$(A_1)$] There are functions $v_k:[0,1]\rightarrow\R$ such that $\lim_{m\to\infty}(Sf^{(m)}_k)=v_k$ pointwise, $1\leq k\leq d$.
\item[$(A_{2,\mathfrak{X}})$] As $m\to\infty$, $\displaystyle(Q_{p_1}(m,f^{(m)}_1,\mathfrak{X}),\dotsc,Q_{p_d}(m,f^{(m)}_d,\mathfrak{X}))^T$ converges in distribution to a $d$-dimensional standard normal vector. 
\item[$(A_{3,\mathfrak{X}})$] The sequences $\displaystyle Q_{p_k}(m,f^{(m)}_k,\mathfrak{X})_{m\in\N}$, $1\leq k\leq d$, satisfy Condition \ref{fmcond}. 
\item[$(A_{4,\mathfrak{X}})$] The above defined sequence $\displaystyle(\mathbf{W}^{(m)}(1),\dotsc, \mathbf{W}^{(m)}(d))^T$, $m\in\N$, of processes based on $\mathfrak{X}$, converges in distribution with respect to the Skorohod topology to a continuous centered Gaussian process $\mathbf{Z}=(\mathbf{Z}(1),\dotsc,\mathbf{Z}(d))^T$ with independent components, whose covariance function is given by \eqref{varlim}.
\item[$(A_5)$] Condition \ref{contrcond} is satisfied.
\item[$(A_6)$] One has $\inf_{i\in\N}\lambda_i>0$.
\end{itemize}
Assume that Assumption $(A_1)$ holds and $p_k\geq 2$ for all $k=1,\dots,d$. Then, the following implications are in order:
\begin{align}\label{e:dagger}
(A_{3,\mathfrak{P}})&\Rightarrow  (A_{2,\mathfrak{G}})\Leftrightarrow(A_{3,\mathfrak{G}})\Leftrightarrow(A_{4,\mathfrak{G}})\Leftrightarrow(A_{5})\\
&\Rightarrow \forall\,\mathfrak{X}\text{ satisfying }\eqref{ufmb}:\bigl(\;(A_{2,\mathfrak{X}})\wedge(A_{3,\mathfrak{X}})\wedge(A_{4,\mathfrak{X}})\bigr).\notag
\end{align}
If, additionally to $(A_1)$, also $(A_6)$ holds and $p_k\geq 2$ for all $k=1,\dots,d$, then we have in fact that 
\begin{align}\label{e:doubledagger}
 (A_{3,\mathfrak{P}})&\Leftrightarrow  (A_{2,\mathfrak{G}})\Leftrightarrow(A_{3,\mathfrak{G}})\Leftrightarrow(A_{4,\mathfrak{G}})\Leftrightarrow(A_{5})
\\
&\Rightarrow \forall\,\mathfrak{X}\text{ satisfying }\eqref{ufmb}:\bigl(\;(A_{2,\mathfrak{X}})\wedge(A_{3,\mathfrak{X}})\wedge(A_{4,\mathfrak{X}})\bigr). \notag
\end{align}
 \end{theorem}

\begin{proof}
Suppose first that $(A_1)$ holds. That $(A_{2,\mathfrak{G}})$, $(A_{3,\mathfrak{G}})$ and $(A_{5})$ are all equivalent is a consequence of \cite[Proposition 1.6]{NPR} (even without assuming $(A_1)$). Of course, $(A_{4,\mathfrak{G}})$ implies  $(A_{2,\mathfrak{G}})$. By Theorem \ref{homsumtheo} applied to $\mathfrak{G}$, under $(A_1)$ we have that $(A_5)$ implies $(A_{4,\mathfrak{G}})$, since $\mathfrak{G}$ obviously satisfies \eqref{ufmb}. Moreover, again by Theorem \ref{homsumtheo} and by Proposition \ref{contrprop}, if the sequence $\mathfrak{X}$ satisfies \eqref{ufmb} and if $(A_{5})$ holds, then (since we have assumed $(A_1)$), we can conclude that $(A_{3,\mathfrak{X}})$ and $(A_{4,\mathfrak{X}})$ hold, immediately implying  $(A_{2,\mathfrak{X}})$ as well. Finally, \cite[Theorem 1.11]{DVZ} ensures that $(A_{3,\mathfrak{P}})$ implies $(A_{2,\mathfrak{G}})$.

Suppose now that both $(A_1)$ and $(A_6)$ hold. Then, \cite[Theorem 1.11]{DVZ} guarantees that $(A_{2,\mathfrak{G}})$ also implies $(A_{3,\mathfrak{P}})$.
\end{proof}

{\red

\begin{remark}[Removing condition $(A_1)$] If Condition $(A_1)$ is removed, then the conclusions of Theorem \ref{universal} continue to hold, provided condition $(A_{4,\mathfrak{X}})$ is replaced by 

\begin{itemize}

\item[$(A'_{4,\mathfrak{X}})$] The collection of distributions of the above defined sequence $\displaystyle(\mathbf{W}^{(m)}(1),\dotsc, \mathbf{W}^{(m)}(d))^T$, $m\in\N$, of processes based on $\mathfrak{X}$, is relatively compact with respect to the Skorohod topology, and every converging subsequence admits as a limit a continuous centered Gaussian process with independent components.
\end{itemize}
The proof of this fact follows the same lines as the proof of Theorem \ref{relcomp} (see Section \ref{ss:proofrelcomp}); the details are left to the reader.
\end{remark}
}

{{}

\begin{example}[Fractional products]\label{ex:fractional} For every integer $p\geq 3$ we will now demonstrate the existence of a sequence of symmetric kernels $f^{(m)} : \{1,...,m\}^p \to \mathbb{R}$, vanishing on diagonals and such that Conditions \ref{contrcond} and $(A_1)$ in Theorem \ref{universal} (for $d=1$) are satisfied, with $v_1(s) = v(s) =  s^{p/a}$, for some integer $2\leq a\leq p-1$. The definition of $f^{(m)}$ is based on a slight variation of the construction of {\it fractional cartesian products}, as described in \cite[Chapter XIII]{Blei-book}, as well as in \cite{BJ, NPR-ejp}. Fix $a,p$ as above, and consider an injective mapping $\varphi : \mathbb{N}^a\mapsto \mathbb{N}$ with the properties that, for all $k\geq 1$, (a)  $\varphi([k]^a) \subset [k^a]$ , and (b) $ \varphi([k+1]^a)\backslash \varphi([k]^a)\subset [(k+1)^a]\backslash [k^a]$. Such a mapping is easily defined by recursion on $k$. Observe in particular that Property (a) implies that, for all $N\geq 1$, the image of the restriction of $\varphi$ to $[\lfloor N^{1/a}\rfloor]^a$ is contained in $[N]$. Consider in addition a collection $\{S_1,...,S_p\}$ of non-empty distinct subsets of $[p]$, satisfying the properties that (i) $|S_i| = a$, and (ii) each index $i\in [p]$ appears in exactly $a$ of the sets $S_i$ (yielding in particular $[p] = \cup_i S_i$). For every $i$ and every vector ${\bf t} = (t_1,...,t_p)$, we set $\pi_{S_i}{\bf t} = (t_k : k\in S_i)$, where the indices belonging to $S_i$ are implicitly listed in increasing order. Our aim is to use the mapping $\varphi$ and the sets $S_i$ in order to define a {\it sparse} subset $F_m$ of $[m]^p$ for every integer $m> p^{a}$ (this last relation is required in order to ensure that $ [ \lfloor m^{1/a}\rfloor ]$ contains at least $p$ elements): the nature of the sparsity of $F_m$ will be encoded by the ratio $p/a$, corresponding to its {\it fractional dimension} \cite{Blei-book, BJ}.  For every $m\geq 1 $ we start by setting
$$
F^0_m := \left\{ {\bf k} = (k_1,...,k_p) : {\bf k} = (\varphi(\pi_{S_1} {\bf t}),..., \varphi(\pi_{S_p} {\bf t}) ), \mbox{for some } {\bf t}\in \Delta^p_{\lfloor m^{1/a}\rfloor} \right\} \subset [m]^p. 
$$
where, the symbol $\Delta_N^p$ indicates the class of all vectors $(t_1,...,t_p)$ such that $1\leq t_1<\cdots < t_p\leq N$; when $m\leq p^a$, $F^0_m$ can be empty. The injectivity of $\varphi$ readily implies that, when $m>p^a$, the entries of any ${\bf k} = (k_1,...,k_p)\in F^0_m$ are pairwise distinct, and also that, if ${\bf k}\neq {\bf k}'$ belong to $F_0^m$, then ${\bf k}$ cannot be obtained from ${\bf k}'$ via a permutation of its entries. We eventually set
$$
F_m : = {\bf sym} (F_m^0)  \subset [m]^p,\quad m\geq 1,
$$
that is, $F_m$ is the class of those $(k_1,...,k_p)$ such that $(k_{\sigma(1)},...,k_{\sigma(p)})\in F_m^0$ for some permutation $\sigma$. Immediate combinatorial considerations imply that there exists a finite constant $b>0$ such that, as $N \to \infty$, $|F_N| \sim b\cdot N^{p/a}$. For $m> p^a$, we define
$$
f ^{(m)} (i_1,...,i_p) := \frac{1}{( p! |F_m|)^{1/2}} {\bf 1}_{F_m} (i_1,...,i_p),\quad (i_1,...i_p) \in [m]^p. 
$$
The computations in \cite[Section 6.2 and Section 6.3]{NPR-ejp} readily imply that $f^{(m)}$ verifies Condition \ref{contrcond}. Moreover, since 
\begin{equation}\label{e:tricky}
|F_m \cap [ \lfloor tm \rfloor ] ^p | \sim |F_{\lfloor tm\rfloor}| \sim b\cdot t^{p/a} m^{p/a},
\end{equation}
 we deduce that Assumption $(A_1)$ in Theorem \ref{universal} is satisfied with $v(t) = t^{p/a}$. Using the chain of implications \eqref{e:dagger}, we infer that, given a sequence $\mathfrak{X}$ of centred unit variance random variables verifying \eqref{ufmb}, the corresponding empirical process $W^{(m)}$, defined via \eqref{defW1}, converges in distribution to $\{ B_{t^{p/a}} : t\in [0,1] \}$, where $B$ is a standard Brownian motion. Selecting pairwise distinct integers $p_1,...,p_d$ yields examples of $d$-dimensional sequences verifying Conditions \ref{contrcond} and $(A_1)$ in Theorem \ref{universal}. For the sake of completeness, a proof of \eqref{e:tricky} is sketched in Section \ref{ss:tricky}.

\end{example}

%\begin{remark} The content 
%
%
%\end{remark}

}

\begin{remark}
In the finite-dimensional setting considered in \cite{NPR} (see, in particular Theorem 1.2 therein), universality of Gaussian homogeneous sums is established under the weaker condition that 
$\sup_{i\in\N} \E|X_i|^3<\infty$ (of course without implying $(A_{3,\mathfrak{X}})$, in this case). This is not possible following our method of proof in the functional situation, since it is based on 
Theorem \ref{maintheo}, which is a genuine fourth moment theorem. On the contrary, the proofs in \cite{NPR} rely on a combination of Malliavin calculus on Gaussian spaces and an invariance principle for multilinear forms, proved in \cite{MDO} via an elaboration of the Lindeberg swapping trick. Since this invariance principle only requires finite third moments, it is possible to dispense with fourth moment conditions on $\mathfrak{X}$ in that framework. {\blue See also \cite[Theorem 4.2]{CSZ17}.    }
\end{remark} 

\begin{remark}
Proposition \ref{contrprop} and Theorems \ref{homsumtheo} and \ref{universal} all require that $p_k\geq 2$ for all $k=1,\dots,d$. This is because display (1.9) in \cite{NPR} used in their proofs relies on this assumption. Relation \eqref{influence}, however, still implies Condition \ref{lindcond}, even if $p_1=1$. In this case, writing $a^{(m)}_{ i}(1)$ for $a^{(m)}_{\lbrace i\rbrace}(1)$, we have $\Inf_i\bigl(f^{(m)}_1\bigr)=a^{(m)}_{i}(1)^2.$ Moreover, a straightforward computation shows that, for $p_1=1$, item a) below implies Condition \ref{fmcond}. 
Therefore, by part b) of Remark \ref{mtrem}, if $p_1=1$ and \eqref{ufmb}  and all of the following hold:
\begin{enumerate}[a)]
\item $\underset{{m\to\infty}}{\lim}\underset{{1\leq i\leq m}}{\max} a^{(m)}_{i}(1)^2=0$;
\item $\underset{{m\to\infty}}\lim\sum_{i=1}^{\lfloor n_ms\rfloor}a_{ i}^{(m)}(1)^2=v_1(s)$ for some $v_1:[0,1]\to\mathbb{R}$;
\item the assumptions of Theorem \ref{homsumtheo} are satisfied for the $(d-1)$-dimensional sequence of processes $\left(Q_{p_2}(m,f^{(m)}_2,\mathfrak{X}),\dots,Q_{p_d}(m,f^{(m)}_d,\mathfrak{X})\right)_{m\in\mathbb{N}}$ substituted in place of $\left(W^{m}\right)_{m\in\mathbb{N}}$, 
\end{enumerate}
 then the whole $d$-dimensional sequence $\left(Q_{p_1}(m,f^{(m)}_1,\mathfrak{X}),\dots,Q_{p_d}(m,f^{(m)}_d,\mathfrak{X})\right)_{m\in\mathbb{N}}$ converges in distribution with respect to the Skorohod topology to the continuous, centered Gaussian process $\mathbf{Z}=(\mathbf{Z}(1),\dotsc,\mathbf{Z}(d))^T$ with independent components, which is defined via \eqref{varlim}.
\end{remark}

\section{Degenerate $U$-statistics and Hoeffding decompositions}\label{Ustats}
In this section we prove and collect useful auxiliary results about degenerate $U$-statistics based on an independent sample. For reasons of simplicity, we abandon the setup of triangular arrays in this section. Thus, we let $(\Omega,\F,\P)$ be a generic probability space, on which %, for some positive integer $N$, 
independent random elements $X_1, X_2\dotsc$ are defined that have values in the respective measurable spaces $(E_1,\mathcal{E}_1), (E_2,\mathcal{E}_2),\dotsc$. Moreover, for $n\in\N$ fixed we let $[n]:=\{1,\dotsc,n\}$ and, for $J\subseteq [n]$, we define $\F_J:=\sigma(X_i,i\in J)$. 
Note that, whenever 
\begin{equation*}
 f:\prod_{j=1}^n E_j\rightarrow\R\quad\text{is}\quad\bigotimes_{j=1}^n\mathcal{E}_j-\B(\R)\text{ - measurable}
\end{equation*}
such that 
\begin{equation*}
Y:=f(X_1,\dotsc,X_n) \in L^1(\Prob)\,,
\end{equation*}
then there is a $\P$-a.s. unique representation of the form
\begin{equation}\label{HDgen}
Y=\sum_{M\subseteq[n]} Y_M =\sum_{s=0}^n\Biggl(\sum_{\substack{M\subseteq[n]:\\ \abs{M}=s}}Y_M\Biggr)
\end{equation}
{where, }for each $M\subseteq [n]$, the summand $Y_M$ is measurable with respect to $\F_M$ and, furthermore, 
\[\E[Y_M\,|\,\F_J]=0\quad\text{holds, whenever }M\nsubseteq J\,.\]
Note that, in particular, there are measurable functions
\begin{equation*}
 f_M:\prod_{j\in M}E_j\rightarrow\R\,,\quad M\subseteq[n]\,,
\end{equation*}
such that $Y_M=f_M(X_j,j\in M)$.
\noindent The representation \eqref{HDgen} is the {celebrated} \textbf{Hoeffding decomposition} of $Y$. The following well-known explicit formula for the \textbf{Hoeffding components} $Y_M$, $M\subseteq[n]$, is {easily deduced from the exclusion-inclusion principle}:
\begin{equation}\label{HDform}
 Y_M=\sum_{J\subseteq M}(-1)^{\abs{M}-\abs{J}}\E[Y\,|\,\F_J]\,,
\end{equation}
yielding in particular that $Y_\emptyset=\E[Y]$ a.s.-$\Prob$. Equation \eqref{HDform} directly implies the linearity of the Hoeffding decomposition. It should also be noted that, whenever $Y\in L^p(\P)$ for some $p\in[1,\infty]$, then automatically $Y_M\in L^p(\P)$ for all $M\subseteq[n]$. 
%For $1\leq p<\infty$, this important fact follows immediately from \eqref{HDform} by an application of the conditional Jensen inequality. If $p=\infty$, this follows from \eqref{HDform} and the monotonicity of conditional expectation.
Moreover, if $Y\in L^2(\P)$, then its Hoeffding components are mutually orthogonal in $L^2(\P)$.

For $p\in[n]$, we call $Y$ a \textbf{(completely) degenerate $U$-statistic of order $p$}, based on $X_1,\dotsc,X_n$, if its Hoeffding decomposition \eqref{HDgen} is of the form 
\begin{equation}\label{degU}
 Y=\sum_{J\in\D_p(n)} Y_J\,,
\end{equation}
i.e. if $Y_M=0$ $\P$-a.s. whenever $|M|\not=p$. Here and in what follows we write 
\begin{equation*}
 \D_p(n):=\{J\subseteq[n]\,:\, |J|=p\}
\end{equation*}
for the collection of all $\binom{n}{p}$ different $p$-subsets of $[n]$.

Now suppose that we are given two positive integers $m$ and $n$, as well as $p\in[n]$ and $q\in[m]$ and consider two degenerate $U$-statistics $V$ based on $X_1,\dotsc,X_m$ and $W$ based on $X_1,\dotsc,X_n$ of respective orders $q$ and $p$. 
Hence, we have the respective Hoeffding decompositions 
\begin{equation*}
 V=\sum_{J\in\D_q(m)}V_J\quad\text{and}\quad W=\sum_{J\in\D_p(n)}W_J\,.
\end{equation*}
We further assume that $\E|W|^4, \E|V|^4<\infty$. Moreover, for $l\in[m]$ and $k\in[n]$ let us define  
\begin{equation*}
 \rho_{l,V}^2:=\max_{1\leq j\leq l}\sum_{\substack{J\in\D_q(l):\\ j\in J}} \E[V_J^2]\quad\text{and}\quad \rho_{k,W}^2:=\max_{1\leq j\leq k}\sum_{\substack{J\in\D_p(k):\\ j\in J}} \E[W_J^2]
\end{equation*}
as well as 
\begin{equation*}
 \sigma_{l,V}^2:=\sum_{J\in\D_q(l)} \E[V_J^2]\quad\text{and}\quad \sigma_{k,W}^2:=\sum_{J\in\D_p(k)} \E[W_J^2]\,.
\end{equation*}
In particular, we have $\sigma_{m,V}^2=\E[V^2]$ and $\sigma_{n,W}^2=\E[W^2]$ as well as the inequalities 
\begin{equation}\label{ineqrhosigma}
\sigma_{k,W}^2\leq\sigma_{n,W}^2\quad\text{and}\quad \rho_{k,W}^2\leq\rho_{n,W}^2
\end{equation}
and analogous ones for $V$ in place of $W$.
Note that $VW$ is an integrable function of $X_1,\dotsc,X_{n\vee m}$. Hence, it follows from \eqref{HDform} and the assumptions on $V$ and $W$ that it has a Hoeffding decomposition of the form 
\begin{equation*}
 VW=\sum_{\substack{M\subseteq[n\vee m]:\\ |M|\leq p+q}} U_M(V,W)\,.
\end{equation*}
For simplicity, we write $U_M(W)$ for $U_M(W,W)$ and $U_M(V)$ for $U_M(V,V)$ so that 
\begin{equation*}
 V^2=\sum_{\substack{M\subseteq[m]:\\ |M|\leq 2q}} U_M(V)\quad\text{and}\quad W^2=\sum_{\substack{N\subseteq[n]:\\ |N|\leq 2p}} U_N(W)
\end{equation*}
are the Hoeffding decompositions of $V^2$ and $W^2$, respectively. We will use the conventions $U_M(V)=0$ and $U_N(W)=0$ implicitly, whenever $M\not\subseteq[m]$ and $N\not\subseteq[n]$, respectively.
Let us define the collection $\mathcal{S}_0:=\mathcal{S}_0(m,n,q,p)$ of all quadruples $(I,J,K,L)\in\D_q(m)^2\times\D_p(n)^2$ such that 
\begin{enumerate}[(i)]
 \item $I\cap K=J\cap L=\emptyset$,
 \item $\emptyset\not=I\cap J=I\setminus(I\cap L)\not=I$,
 \item $\emptyset\not=J\cap I=J\setminus(J\cap K)\not=J$,
 \item $\emptyset\not=K\cap J=K\setminus(L\cap K)\not=K$ and 
 \item $\emptyset\not=L\cap I=L\setminus(L\cap K)\not=L$.
\end{enumerate}
Similarly, we denote by $\mathcal{S}_1:=\mathcal{S}_1(m,n,p)$ the collection of all quadruples $(I,J,K,L)\in\D_p(m\wedge n)^4$ such that 
\begin{enumerate}[(i)]
 \item $I\cap J=K\cap L=\emptyset$,
 \item $\emptyset\not=I\cap K=I\setminus(I\cap L)\not=I$,
 \item $\emptyset\not=J\cap K=J\setminus(J\cap L)\not=J$,
 \item $\emptyset\not=K\cap J=K\setminus(I\cap K)\not=K$ and 
 \item $\emptyset\not=L\cap I=L\setminus(L\cap J)\not=L$.
\end{enumerate}

Moreover, we define the quantities
\begin{align*}
S_0(V,W)&=\sum_{\substack{(I,J,K,L)\in\mathcal{S}_0}} \E\bigl[V_IV_JW_KW_L\bigr]\quad\text{and, if } p=q\,,\\ 
%\sum_{\substack{I,J\in\D_q(m),\\ K,L\in\D_p(n):\\ I\cap K= J\cap L=\emptyset,\\ \emptyset\not=I\cap J=I\setminus(I\cap L)\not=I,\\ \emptyset\not=J\cap I=J\setminus(J\cap K)\not=J}} \E\bigl[V_IV_JW_KW_L\bigr]\,,\quad\text{and}\\ 
%\sum_{\substack{I,J,K,L\in\D_p(m\wedge n):\\ I\cap K= J\cap L=\emptyset,\\ \emptyset\not=I\cap J=I\setminus(I\cap L)\not=I,\\ \emptyset\not=J\cap I=J\setminus(J\cap K)\not=J}} \E\bigl[V_IV_JW_KW_L\bigr]\,,\quad\text{and}\\ 
S_1(V,W)&=\sum_{\substack{(I,J,K,L)\in\mathcal{S}_1}}\E\bigl[V_IV_JW_KW_L\bigr]\,.
%=\sum_{\substack{I,J,K,L\in\D_p(m\wedge n):\\ I\cap J= K\cap L=\emptyset,\\ \emptyset\not=I\cap K=I\setminus(I\cap L)\not=I,\\ \emptyset\not=J\cap L=J\setminus(J\cap K)\not=J}} \E\bigl[V_IV_JW_KW_L\bigr]\,.
%&=\sum_{\substack{I,J,K,L\in\D_p(m\wedge n):\\ I\cap K= J\cap L=\emptyset,\\ \emptyset\not=I\cap J=I\setminus(I\cap L)\not=I,\\ \emptyset\not=J\cap I=J\setminus(J\cap K)\not=J}} \E\bigl[V_IW_JV_KW_L\bigr]\,.
\end{align*}
In particular, we have $S_0(W,W)=S_1(W,W)$. More generally, if $p=q$, $m\leq n$ and $V_J=W_J$ for all $J\in\D_p(m)$, then $S_1(V,W)=S_0(V,V)$.
{
\begin{remark}
The quadruples $(I,J,K,L)\in\mathcal{S}_0$ play a crucial role in all proofs of de Jong type CLTs via the quantities $S_0(V,W)$ (see \cite{deJo90, DP17}). Note that these quadruples are all \textit{bifold}, meaning that each element $i\in I\cup J\cup K\cup L$ lies in exactly two of the sets $I,J,K,L$. In particular, there can be no \textit{free index}, which is an element that appears in only one of the sets $I,J,K,L$. Moreover, no two of the sets $I,J,K,L$ can be the same and each of these sets is disjoint to precisely one other among these sets. An analogous remark applies to the quadruples $(I,J,K,L)\in\mathcal{S}_1$. Note further that the existence of a free index for a quadruple $(I,J,K,L)$ necessarily implies $\E[V_IV_JW_KW_L]=0$ by degeneracy. This rule will be tacitly applied in what follows.    
\end{remark}
}
\begin{lemma}\label{s0lemma}
There exists a finite constant $C_{p,q}$, only depending on $p$ and $q$, such that 
\[S_0(V,W)\geq -C_{p,q}\max\bigl(\sigma_{m,V}^2\rho_{n,W}^2,\sigma_{n,W}^2\rho_{m,V}^2\bigr)\,.\]
\end{lemma}

\begin{proof}
This follows from a straightforward generalization of Propositions 3.5 and 3.6 of \cite{DP17} to the situation of possibly different $m$ and $n$ and possibly non-unit variances .
We omit the details here.
\end{proof}

\begin{lemma}\label{s0lemma2}
We have that
\[S_0(W,W)\leq \E[W^4]-3\sigma_{n,W}^4+2p\sigma_{n,W}^2\rho_{n,W}^2  \,.\]
\end{lemma}

\begin{proof}
This follows from the first inequality in equation (2.13) of \cite{DP17} by noticing that the left hand side of this inequality is nonnegative and by taking into account that $\Var(W)=\sigma_{n,W}^2$ is not necessarily equal to $1$, here.
\end{proof}

The next two lemmas will be of vital importance for proving convergence of finite dimensional distributions. Lemma \ref{varlemma2} is an extension of results and techniques provided in \cite[Section 3]{DP17}, whereas Lemma \ref{varlemma1} is taken directly from there.
\begin{lemma}[Lemma 2.10 of \cite{DP17}]\label{varlemma1}
There is a finite constant $\kappa_p$ only depending on $p$ such that
\begin{align*}
\sum_{\substack{M\subseteq[n]:\\ 1\leq|M|\leq 2p-1}}\Var\bigl(U_M(W)\bigr)\leq \E[W^4]-3\sigma_{n,W}^4 +\kappa_p \sigma_{n,W}^2\rho_{n,W}^2\,.
\end{align*}
\end{lemma}

\begin{lemma}\label{varlemma2}
 \begin{enumerate}[{\normalfont (i)}]
  \item If $p=q$, then 
  \begin{align*}
 &  \sum_{\substack{M\subseteq[n\vee m]:\\ |M|\leq 2p-1}} \Var\bigl(U_M(V,W)\bigr)\leq p\min\bigl(\rho_{n,W}^2\sigma_{m,V}^2,\,\rho_{m,V}^2\sigma_{n,W}^2\bigr)+ p\rho_{m\wedge n,V}\rho_{m\wedge n,W}\sigma_{m\wedge n,V}\sigma_{m\wedge n,W}\\
 &\;\hspace{1.5cm}+\biggl(\sum_{\substack{M\subseteq[m\wedge n]:\\ 1\leq|M|\leq 2p-1}}\Var\bigl(U_M(V)\bigr)\biggr)^{1/2}\cdot \biggl(\sum_{\substack{M\subseteq[m\wedge n]:\\ 1\leq|M|\leq 2p-1}}\Var\bigl(U_M(W)\bigr)\biggr)^{1/2}+S_1(V,W)-S_0(V,W)\,.
  \end{align*}
\item If $p\not=q$, then 
\begin{align*}
 &\sum_{\substack{M\subseteq[n\vee m]:\\ |M|\leq p+q-1}} \Var\bigl(U_M(V,W)\bigr)\leq\min\bigl(q\rho_{n,W}^2\sigma_{m,V}^2,\,p\rho_{m,V}^2\sigma_{n,W}^2\bigr)- S_0(V,W)\\
% +C_{p,q}\max\bigl(\sigma_{m,V}^2\rho_{n,W}^2,\sigma_{n,W}^2\rho_{m,V}^2\bigr)\\
&\;\hspace{4cm}+ \biggl(\sum_{\substack{M\subseteq[m\wedge n]:\\ 1\leq|M|\leq 2(p\vee q)-1}}\Var\bigl(U_M(V)\bigr)\biggr)^{1/2} \cdot \biggl(\sum_{\substack{M\subseteq[m\wedge n]:\\ 1\leq|M|\leq 2(p\vee q)-1}}\Var\bigl(U_M(W)\bigr)\biggr)^{1/2}\,.%\\
%&\leq\min\bigl(q\rho_{n,W}^2\sigma_{m,V}^2,\,p\rho_{m,V}^2\sigma_{n,W}^2\bigr) +C_{p,q}\max\bigl(\sigma_{m,V}^2\rho_{n,W}^2,\sigma_{n,W}^2\rho_{m,V}^2\bigr)\\
%&\;+ \biggl(\sum_{\substack{M\subseteq[m\wedge n]:\\ 1\leq|M|\leq 2(p\vee q)-1}}\Var\bigl(U_M(V)\bigr)\biggr)^{1/2} \cdot \biggl(\sum_{\substack{M\subseteq[m\wedge n]:\\ 1\leq|M|\leq 2(p\vee q)-1}}\Var\bigl(U_M(W)\bigr)\biggr)^{1/2}\,,
\end{align*}
 \end{enumerate}

\end{lemma}

\begin{proof}
 We begin with some computations that are valid in all different cases. 
 \begin{align*}
  &\sum_{\substack{M\subseteq[n\vee m]:\\ |M|\leq p+q-1}} \Var\bigl(U_M(V,W)\bigr)=\Var(VW)-\sum_{\substack{M\subseteq[n\vee m]:\\ |M|= p+q}} \Var\bigl(U_M(V,W)\bigr)\\
&=\Var(VW)-\sum_{\substack{M\subseteq[n\vee m]:\\ |M|= p+q}}\Var\Bigl(\sum_{\substack{J\in\D_q(m), K\in \D_p(n):\\ J\cup K=M}}V_J W_K\Bigr)\\
&=\Var(VW)-\sum_{\substack{M\subseteq[n\vee m]:\\ |M|= p+q}}\sum_{\substack{I,J\in\D_q(m),\\ K,L\in\D_p(n):\\ I\cap K= J\cap L=\emptyset,\\I\cup K=M=J\cup L}}\E\bigl[V_I V_J W_K W_L\bigr]\\
&=\Var(VW)-\sum_{\substack{I,J\in\D_q(m),\\ K,L\in\D_p(n):\\ I\cap K= J\cap L=\emptyset}}\E\bigl[V_I V_J W_K W_L\bigr]\\
&=\Var(VW)-\sum_{\substack{I\in\D_q(m),\\ K\in\D_p(n):\\ I\cap K=\emptyset}}\E\bigl[V_I^2\bigr]\E\bigl[W_K^2\bigr]-
\sum_{\substack{I,J\in\D_q(m),\\ K,L\in\D_p(n):\\ I\cap K= J\cap L=\emptyset,\\ \emptyset\not=I\cap J=I\setminus(I\cap L)\not=I,\\ \emptyset\not=J\cap I=J\setminus(J\cap K)\not=J}} \E\bigl[V_IV_JW_KW_L\bigr]\\
&\hspace{3cm}-\delta_{p,q}\sum _{\substack{I,K\in\D_q(m\wedge n):\\ I\cap K=\emptyset}}\E\bigl[V_IW_I\bigr]\E\bigl[V_KW_K\bigr]
 \end{align*}
Noting that 
\begin{align*}
 \Var(VW)&=\Cov(V^2,W^2)+\E[V^2]\E[W^2]-\E[VW]^2
\end{align*}
and also, for reasons of degeneracy, that 
\begin{align*}
 \sum_{\substack{I,J\in\D_q(m),\\ K,L\in\D_p(n):\\ I\cap K= J\cap L=\emptyset,\\ \emptyset\not=I\cap J=I\setminus(I\cap L)\not=I,\\ \emptyset\not=J\cap I=J\setminus(J\cap K)\not=J}} \E\bigl[V_IV_JW_KW_L\bigr]&=S_0(V,W)
 % =\sum_{\substack{I,J\in\D_q(m),\\ K,L\in\D_p(n):\\ I\cap K= J\cap L=\emptyset,\\ \emptyset\not=I\cap J=I\setminus(I\cap L)\not=I}} \E\bigl[V_IV_JW_KW_L\bigr]
\end{align*}
we hence obtain that 
\begin{align}\label{vl1}
 &\sum_{\substack{M\subseteq[n\vee m]:\\ |M|\leq p+q-1}} \Var\bigl(U_M(V,W)\bigr)
 %= \Cov(V^2,W^2)+\E[V^2]\E[W^2]-\E[VW]^2\notag\\
%&\; -\sum_{\substack{I\in\D_q(m),\\ K\in\D_p(n):\\ I\cap K=\emptyset}}\E\bigl[V_I^2\bigr]\E\bigl[W_K^2\bigr]
% -\sum_{\substack{I,J\in\D_q(m),\\ K,L\in\D_p(n):\\ I\cap K= J\cap L=\emptyset,\\ \emptyset\not=I\cap J=I\setminus(I\cap L)\not=I}} \E\bigl[V_IV_JW_KW_L\bigr]\notag\\
%&\; -\delta_{p,q}\sum _{\substack{I\in\D_q(m),\\ K\in\D_p(n):\\ I\cap K=\emptyset}}\E\bigl[V_IW_I\bigr]\E\bigl[V_KW_K\bigr]\notag\\
=\Cov(V^2,W^2)+\E[V^2]\E[W^2]-\E[VW]^2 \notag\\
&\;\hspace{2cm}-\sum_{\substack{I\in\D_q(m),\\ K\in\D_p(n):\\ I\cap K=\emptyset}}\E\bigl[V_I^2\bigr]\E\bigl[W_K^2\bigr]-S_0(V,W)-\delta_{p,q}\sum _{\substack{I\in\D_q(m),\\ K\in\D_p(n):\\ I\cap K=\emptyset}}\E\bigl[V_IW_I\bigr]\E\bigl[V_KW_K\bigr]\,.
\end{align}

We now deal with the individual terms appearing on the right hand side of \eqref{vl1}.
\begin{align*}
 \sum_{\substack{I\in\D_q(m),\\ K\in\D_p(n):\\ I\cap K=\emptyset}}\E\bigl[V_I^2\bigr]\E\bigl[W_K^2\bigr]&=\sum_{\substack{I\in\D_q(m),\\ K\in\D_p(n)}}\E\bigl[V_I^2\bigr]\E\bigl[W_K^2\bigr]
 -\sum_{\substack{I\in\D_q(m),\\ K\in\D_p(n):\\ I\cap K\not=\emptyset}}\E\bigl[V_I^2\bigr]\E\bigl[W_K^2\bigr]\\
 &=\E[V^2]\E[W^2]-\sum_{\substack{I\in\D_q(m),\\ K\in\D_p(n):\\ I\cap K\not=\emptyset}}\E\bigl[V_I^2\bigr]\E\bigl[W_K^2\bigr]\,.
\end{align*}
Hence,
\begin{align*}
 &\sum_{\substack{M\subseteq[n\vee m]:\\ |M|\leq p+q-1}} \Var\bigl(U_M(V,W)\bigr)= \Cov(V^2,W^2)-\E[VW]^2\notag\\
 &\hspace{2cm}+\sum_{\substack{I\in\D_q(m),\\ K\in\D_p(n):\\ I\cap K\not=\emptyset}}\E\bigl[V_I^2\bigr]\E\bigl[W_K^2\bigr]
 -S_0(V,W)-\delta_{p,q}\sum _{\substack{I,K\in\D_q(m\wedge n):\\ I\cap K=\emptyset}}\E\bigl[V_IW_I\bigr]\E\bigl[V_KW_K\bigr]\,.
\end{align*}
Moreover, if $p=q$ then 
\begin{align*}
&\sum _{\substack{I,K\in\D_p(m\wedge n):\\ I\cap K=\emptyset}}\E\bigl[V_IW_I\bigr]\E\bigl[V_KW_K\bigr]\\
&= \biggl(\sum_{I\in\D_p(m\wedge n)}\E\bigl[V_IW_I\bigr]\biggr)^2-\sum _{\substack{I,K\in\D_p(m\wedge n):\\ I\cap K\not=\emptyset}}\E\bigl[V_IW_I\bigr]\E\bigl[V_KW_K\bigr]\\
&=\E[VW]^2-\sum _{\substack{I,K\in\D_p(m\wedge n):\\ I\cap K\not=\emptyset}}\E\bigl[V_IW_I\bigr]\E\bigl[V_KW_K\bigr]\,.
\end{align*}
Hence, using that $\E[VW]=0$ if $p\not=q$, we obtain that 
\begin{align}\label{vl8}
 &\sum_{\substack{M\subseteq[n\vee m]:\\ |M|\leq p+q-1}} \Var\bigl(U_M(V,W)\bigr)= \Cov(V^2,W^2)-2\delta_{p,q}\E[VW]^2\notag\\
 &\hspace{1cm}+\sum_{\substack{I\in\D_q(m),\\ K\in\D_p(n):\\ I\cap K\not=\emptyset}}\E\bigl[V_I^2\bigr]\E\bigl[W_K^2\bigr] 
  -S_0(V,W)+\delta_{p,q}\sum _{\substack{I,K\in\D_p(m\wedge n):\\ I\cap K\not=\emptyset}}\E\bigl[V_IW_I\bigr]\E\bigl[V_KW_K\bigr]\,.
\end{align}

In particular, if $p\not=q$, then 
\begin{align}\label{pnotq}
  &\sum_{\substack{M\subseteq[n\vee m]:\\ |M|\leq p+q-1}} \Var\bigl(U_M(V,W)\bigr)= \Cov(V^2,W^2)+\sum_{\substack{I\in\D_q(m),\\ K\in\D_p(n):\\ I\cap K\not=\emptyset}}\E\bigl[V_I^2\bigr]\E\bigl[W_K^2\bigr]
 -S_0(V,W)\,.
\end{align}

Now, 
\begin{align*}
 \sum_{\substack{I\in\D_q(m),\\ K\in\D_p(n):\\ I\cap K\not=\emptyset}}\E\bigl[V_I^2\bigr]\E\bigl[W_K^2\bigr]&=\sum_{I\in\D_q(m)}\E[V_I^2]\sum_{i\in I}\sum_{\substack{K\in\D_p(n):\\\max(K\cap I)=i}}\E[W_K^2]\\
 &\leq q\rho_{n,W}^2\sum_{I\in\D_q(m)}\E[V_I^2]=q\rho_{n,W}^2\E[V]^2=q\rho_{n,W}^2\sigma_{m,V}^2\,,
\end{align*}
and, analogously, 
\begin{align*}
 \sum_{\substack{I\in\D_q(m),\\ K\in\D_p(n):\\ I\cap K\not=\emptyset}}\E\bigl[V_I^2\bigr]\E\bigl[W_K^2\bigr]&\leq p\rho_{m,V}^2\E[W]^2=p\rho_{m,V}^2\sigma_{n,W}^2\,.
\end{align*}
Hence, 
\begin{align}\label{vl2}
 \sum_{\substack{I\in\D_q(m),\\ K\in\D_p(n):\\ I\cap K\not=\emptyset}}\E\bigl[V_I^2\bigr]\E\bigl[W_K^2\bigr]&\leq\min\bigl(q\rho_{n,W}^2\sigma_{m,V}^2,\,p\rho_{m,V}^2\sigma_{n,W}^2\bigr)\,.
\end{align}

We have 
\begin{align*}
 &\Babs{\sum _{\substack{I,K\in\D_p(m\wedge n):\\ I\cap K\not=\emptyset}}\E\bigl[V_IW_I\bigr]\E\bigl[V_KW_K\bigr]}\leq \sum_{I\in\D_p(m\wedge n)}\babs{\E\bigl[V_IW_I\bigr]}
 \sum_{\substack{K\in\D_p(m\wedge n):\\ I\cap K\not=\emptyset}}\babs{\E\bigl[V_KW_K\bigr]}\\
 &=\sum_{I\in\D_p(m\wedge n)}\babs{\E\bigl[V_IW_I\bigr]}\sum_{i\in I}\sum_{\substack{K\in\D_p(m\wedge n):\\ \max(I\cap K)=i}}\babs{\E\bigl[V_KW_K\bigr]}\\
 &\leq \sum_{I\in\D_p(m\wedge n)}\babs{\E\bigl[V_IW_I\bigr]}\sum_{i\in I}\Bigl(\sum_{\substack{K\in\D_p(m\wedge n):\\ \max(I\cap K)=i}}\E[V_K^2]\Bigr)^{1/2}\cdot \Bigl(\sum_{\substack{K\in\D_p(m\wedge n):\\ \max(I\cap K)=i}}\E[W_K^2]\Bigr)^{1/2}\\
 &\leq p\rho_{m\wedge n,V}\rho_{m\wedge n,W}\sum_{I\in\D_p(m\wedge n)}\babs{\E\bigl[V_IW_I\bigr]}\\
 &\leq  p\rho_{m\wedge n,V}\rho_{m\wedge n,W}\Bigl(\sum_{I\in\D_p(m\wedge n)}\E[V_I^2]\Bigr)^{1/2}\cdot \Bigl(\sum_{I\in\D_p(m\wedge n)}\E[W_I^2]\Bigr)^{1/2}\\
 &=p\rho_{m\wedge n,V}\rho_{m\wedge n,W}\sigma_{m\wedge n,V}\sigma_{m\wedge n,W}\,.
\end{align*}
Hence, we obtain the inequalities
\begin{align}\label{vl3}
 -p\rho_{m\wedge n,V}\rho_{m\wedge n,W}\sigma_{m\wedge n,V}\sigma_{m\wedge n,W}&\leq \sum _{\substack{I\in\D_p(m),\\ K\in\D_p(n):\\ I\cap K\not=\emptyset}}\E\bigl[V_IW_I\bigr]\E\bigl[V_KW_K\bigr]\notag\\
& \leq p\rho_{m\wedge n,V}\rho_{m\wedge n,W}\sigma_{m\wedge n,V}\sigma_{m\wedge n,W}
\end{align}

We estimate the term $\Cov(V^2,W^2)$ seperately according to whether $p=q$ or not. First assume that $p=q$. Then, due to the orthogonality of Hoeffding components we have 
\begin{align*}
 \Cov(V^2,W^2)&=\sum_{\substack{M\subseteq[m],\\ N\subseteq[n]:\\ 1\leq|M|,|N|\leq 2p}}\E\bigl[U_M(V) U_N(W)\bigr]=\sum_{\substack{M\subseteq[m\wedge n]:\\ 1\leq|M|\leq 2p}}\E\bigl[U_M(V) U_M(W)\bigr]\\
 &=\sum_{\substack{M\subseteq[m\wedge n]:\\ 1\leq|M|\leq 2p-1}}\E\bigl[U_M(V) U_M(W)\bigr]+\sum_{\substack{M\subseteq[m\wedge n]:\\ |M|=2p}}\E\bigl[U_M(V) U_M(W)\bigr]\\
 &=:T_1+T_2\,.
\end{align*}
Using the Cauchy-Schwarz inequality twice we can estimate
\begin{align*}
 \abs{T_1}&\leq \sum_{\substack{M\subseteq[m\wedge n]:\\ 1\leq|M|\leq 2p-1}}\sqrt{\Var\bigl(U_M(V)\bigr)}\sqrt{\Var\bigl(U_M(W)\bigr)}\\
& \leq\biggl(\sum_{\substack{M\subseteq[m\wedge n]:\\ 1\leq|M|\leq 2p-1}}\Var\bigl(U_M(V)\bigr)\biggr)^{1/2}\cdot \biggl(\sum_{\substack{M\subseteq[m\wedge n]:\\ 1\leq|M|\leq 2p-1}}\Var\bigl(U_M(W)\bigr)\biggr)^{1/2}\,.
\end{align*}
On the other hand, we have 
\begin{align*}
 T_2&=\sum_{\substack{M\subseteq[m\wedge n]:\\ |M|=2p}}\sum_{\substack{I,J\in\D_p(m),\\ K,L\in\D_p(n):\\ I\cup J=K\cup L=M}}\E\bigl[V_IV_JW_KW_L\bigr]
 =\sum_{\substack{I,J\in\D_p(m),\\ K,L\in\D_p(n):\\ I\cap J=K\cap L=\emptyset,\\I\cup J=K\cup L }}\E\bigl[V_IV_JW_KW_L\bigr]\\
 &=\sum_{\substack{I,J\in\D_p(m),\\ K,L\in\D_p(n):\\ I\cap J=K\cap L=\emptyset}}\E\bigl[V_IV_JW_KW_L\bigr] \\
 &=2\sum _{\substack{I,J\in\D_p(m\wedge n):\\ I\cap J=\emptyset}}\E\bigl[V_IW_I\bigr]\E\bigl[V_JW_J\bigr] 
 + \sum_{\substack{I,J\in\D_p(m),\\ K,L\in\D_p(n):\\ I\cap J= K\cap L=\emptyset,\\ \emptyset\not=I\cap K=I\setminus(I\cap L)\not=I}} \E\bigl[V_IV_JW_KW_L\bigr]\\
& =2\sum _{\substack{I,J\in\D_p(m\wedge n):\\ I\cap J=\emptyset}}\E\bigl[V_IW_I\bigr]\E\bigl[V_JW_J\bigr] +S_1(V,W)\\
 &=2\E[VW]^2-2 \sum _{\substack{I,J\in\D_p(m\wedge n):\\ I\cap J\not=\emptyset}}\E\bigl[V_IW_I\bigr]\E\bigl[V_JW_J\bigr] +S_1(V,W)    \,,
\end{align*}
where we have used that, due to degeneracy,
\begin{align*}
&\sum_{\substack{I,J\in\D_p(m),\\ K,L\in\D_p(n):\\ I\cap J= K\cap L=\emptyset,\\ \emptyset\not=I\cap K=I\setminus(I\cap L)\not=I}} \E\bigl[V_IV_JW_KW_L\bigr]
=\sum_{\substack{I,J,K,L\in\D_p(m\wedge n):\\ I\cap J= K\cap L=\emptyset,\\ \emptyset\not=I\cap K=I\setminus(I\cap L)\not=I}} \E\bigl[V_IV_JW_KW_L\bigr]\\
&=\sum_{\substack{I,J,K,L\in\D_p(m\wedge n):\\ I\cap J= K\cap L=\emptyset,\\ \emptyset\not=I\cap K=I\setminus(I\cap L)\not=I,\\ \emptyset\not=J\cap L=J\setminus(J\cap K)\not=J}} \E\bigl[V_IV_JW_KW_L\bigr]=S_1(V,W)\,.
% S_1(V,W)&=\sum_{\substack{I,J\in\D_p(m),\\ K,L\in\D_p(n):\\ I\cap J= K\cap L=\emptyset,\\ \emptyset\not=I\cap K=I\setminus(I\cap L)\not=I,\\ \emptyset\not=J\cap L=J\setminus(J\cap K)\not=J}} \E\bigl[V_IV_JW_KW_L\bigr]
% =\sum_{\substack{I,J\in\D_p(m),\\ K,L\in\D_p(n):\\ I\cap J= K\cap L=\emptyset,\\ \emptyset\not=I\cap K=I\setminus(I\cap L)\not=I}} \E\bigl[V_IV_JW_KW_L\bigr]\\
% &=\sum_{\substack{I,J,K,L\in\D_p(m\wedge n):\\ I\cap J= K\cap L=\emptyset,\\ \emptyset\not=I\cap K=I\setminus(I\cap L)\not=I}} \E\bigl[V_IV_JW_KW_L\bigr] \,.
\end{align*}
%{\blue This is not quite the same quantity as $S_0(V,W)$ unless $V=W$!!% Also note that $\E\bigl[V_IV_JW_KW_L\bigr]=0$ unless $I,J,K,L\subseteq[m\wedge n]$.
%}\\
\begin{comment}
If we denote by $J_k$ the projection on the $k$-th Hoeffding space, then 
\begin{equation*}
 T_2=\E\bigl[J_{2p}(W^2)J_{2p}(V^2)\bigr]\,.
\end{equation*}
Since 
\begin{equation*}
 J_{2p}(W^2)=\sum_{\substack{I,J\in\D_p(n):\\ I\cap J=\emptyset}} W_I W_J\quad\text{and}\quad J_{2p}(V^2)=\sum_{\substack{K,L\in\D_p(m):\\ K\cap L=\emptyset}} V_K V_L\,,
\end{equation*}
we obtain that 
\begin{align}\label{vl7}
 T_2&=\E\bigl[J_{2p}(W^2)J_{2p}(V^2)\bigr]=\sum_{\substack{I,J\in\D_p(m),\\ K,L\in\D_p(n):\\ I\cap J=K\cap L=\emptyset}}\E\bigl[V_IV_JW_KW_L\bigr]\notag\\
 &= \sum_{\substack{I,J,K,L\in\D_p(m\wedge n):\\ I\cap J=K\cap L=\emptyset}}\E\bigl[V_IV_JW_KW_L\bigr]\notag\\
 & =2\sum _{\substack{I,K\in\D_p(m\wedge n):\\ I\cap K=\emptyset}}\E\bigl[V_IW_I\bigr]\E\bigl[V_KW_K\bigr] +S_1(V,W)\notag\\
 &=2\E[VW]^2-2 \sum _{\substack{I,K\in\D_p(m\wedge n):\\ I\cap K\not=\emptyset}}\E\bigl[V_IW_I\bigr]\E\bigl[V_KW_K\bigr] +S_1(V,W)  \,.
\end{align}
\end{comment}
Hence, if $p=q$ then we obtain that
\begin{align}\label{vl4}
 \Cov(V^2,W^2)
% &\leq \biggl(\sum_{\substack{M\subseteq[m\wedge n]:\\ 1\leq|M|\leq 2p-1}}\Var\bigl(U_M(V)\bigr)\biggr)^{1/2}\cdot \biggl(\sum_{\substack{M\subseteq[m\wedge n]:\\ 1\leq|M|\leq 2p-1}}\Var\bigl(U_M(W)\bigr)\biggr)^{1/2}\notag\\
%&\; +2\sum _{\substack{I,K\in\D_p(m\wedge n):\\ I\cap K=\emptyset}}\E\bigl[V_IW_I\bigr]\E\bigl[V_KW_K\bigr] +S_1(V,W)\notag\\
&\leq\biggl(\sum_{\substack{M\subseteq[m\wedge n]:\\ 1\leq|M|\leq 2p-1}}\Var\bigl(U_M(V)\bigr)\biggr)^{1/2}\cdot \biggl(\sum_{\substack{M\subseteq[m\wedge n]:\\ 1\leq|M|\leq 2p-1}}\Var\bigl(U_M(W)\bigr)\biggr)^{1/2}\notag\\
&\;+2\E[VW]^2-2 \sum _{\substack{I,K\in\D_p(m\wedge n):\\ I\cap K\not=\emptyset}}\E\bigl[V_IW_I\bigr]\E\bigl[V_KW_K\bigr] +S_1(V,W)\,.
\end{align}
Altogether, in the case $p=q$, we obtain from \eqref{vl8}-\eqref{vl4} that 
\begin{align*}
 &\sum_{\substack{M\subseteq[n\vee m]:\\ |M|\leq p+q-1}} \Var\bigl(U_M(V,W)\bigr)\leq p\min\bigl(\rho_{n,W}^2\sigma_{m,V}^2,\,\rho_{m,V}^2\sigma_{n,W}^2\bigr) \\
&\;+ p\rho_{m\wedge n,V}\rho_{m\wedge n,W}\sigma_{m\wedge n,V}\sigma_{m\wedge n,W} +S_1(V,W)-S_0(V,W)\\
 &\;+\biggl(\sum_{\substack{M\subseteq[m\wedge n]:\\ 1\leq|M|\leq 2p-1}}\Var\bigl(U_M(V)\bigr)\biggr)^{1/2}\cdot \biggl(\sum_{\substack{M\subseteq[m\wedge n]:\\ 1\leq|M|\leq 2p-1}}\Var\bigl(U_M(W)\bigr)\biggr)^{1/2}\,,
\end{align*}
proving part (i). Let us now assume that $p\not=q$. Let $r:=p\wedge q< p\vee q=:s$. In this case, clearly, $\E[VW]=0$. Moreover, similarly as before, we have 
\begin{align}\label{vl5}
 \Cov(V^2,W^2)&=\sum_{\substack{M\subseteq[m],\, N\subseteq[n]:\\ 1\leq|M|\leq 2q,\\1\leq |N|\leq 2p}}\E\bigl[U_M(V) U_N(W)\bigr]=\sum_{\substack{M\subseteq[m\wedge n]:\\ 1\leq|M|\leq 2r}}\E\bigl[U_M(V) U_M(W)\bigr]\notag\\
 &=\sum_{\substack{M\subseteq[m\wedge n]:\\ 1\leq|M|\leq 2s-1}}\E\bigl[U_M(V) U_M(W)\bigr]\notag\\
 &\leq \sum_{\substack{M\subseteq[m\wedge n]:\\ 1\leq|M|\leq 2s-1}}\sqrt{\Var\bigl(U_M(V)\bigr)}\sqrt{\Var\bigl(U_M(W)\bigr)}\notag\\
 &\leq \biggl(\sum_{\substack{M\subseteq[m\wedge n]:\\ 1\leq|M|\leq 2s-1}}\Var\bigl(U_M(V)\bigr)\biggr)^{1/2} \cdot \biggl(\sum_{\substack{M\subseteq[m\wedge n]:\\ 1\leq|M|\leq 2s-1}}\Var\bigl(U_M(W)\bigr)\biggr)^{1/2}\,.
 \end{align}
This finishes the proof of (ii).
\begin{comment}
Hence, the first inequality in (ii) follows from \eqref{pnotq}, \eqref{vl2} and \eqref{vl5}. 
To obtain the second inequality,
it remains to bound the quantity $S_0(V,W)$. A slight generalization of formula (3.18) of \cite{DP17} yields the bound 
\begin{equation}\label{vl6}
S_0(V,W)\geq - C_{p,q}\max\bigl(\sigma_{m,V}^2\rho_{n,W}^2,\sigma_{n,W}^2\rho_{m,V}^2\bigr)\,,
\end{equation}
where $C_{p,q}$ is a finite, positive constant that only depends on $p$ and $q$. The generalization with respect to \cite{DP17} is that, here, we allow $n$ and $m$ to be different and that the variances of $V$ and $W$ are not assumed to be 
equal to $1$. Now, \eqref{vl1},\eqref{vl2}, \eqref{vl5} and \eqref{vl6} imply that
\begin{align*}
&\sum_{\substack{M\subseteq[n\vee m]:\\ |M|\leq p+q-1}} \Var\bigl(U_M(V,W)\bigr)\leq \min\bigl(q\rho_{n,W}^2\sigma_{m,V}^2,\,p\rho_{m,V}^2\sigma_{n,W}^2\bigr)\\
&\;+C_{p,q}\max\bigl(\sigma_{m,V}^2\rho_{n,W}^2,\sigma_{n,W}^2\rho_{m,V}^2\bigr)\\
&\;+ \biggl(\sum_{\substack{M\subseteq[m\wedge n]:\\ 1\leq|M|\leq 2s-1}}\Var\bigl(U_M(V)\bigr)\biggr)^{1/2} \cdot \biggl(\sum_{\substack{M\subseteq[m\wedge n]:\\ 1\leq|M|\leq 2s-1}}\Var\bigl(U_M(W)\bigr)\biggr)^{1/2}\,,
\end{align*}
which proves (ii).
\end{comment}
\end{proof}

\begin{corollary}\label{varcor1}
In the situation of Lemma \ref{varlemma2} (i) assume that, additionally, $m\leq n$ and $V_J=W_J$ hold for all $J\in\D_p(m)$. Then, we have 
\begin{align*}
 &  \sum_{\substack{M\subseteq[n\vee m]:\\ |M|\leq 2p-1}} \Var\bigl(U_M(V,W)\bigr)\leq \E[V^4]-3\sigma_{m,V}^4+2p\sigma_{m,V}^2\rho_{m,V}^2 +\bigl(C_{p,q}+2p \bigr)\rho_{n,W}^2\sigma_{n,W}^2\\
&\;+\biggl(\E[V^4]-3\sigma_{m,V}^4+\kappa_p\sigma_{m,V}^2\rho_{m,V}^2\biggr)^{1/2}\cdot \biggl(\E[W^4]-3\sigma_{n,W}^4+\kappa_p\sigma_{n,W}^2\rho_{n,W}^2\biggr)^{1/2}\,
\end{align*}
where $\kappa_p$ is a finite constant depending only on $p$.
\end{corollary}

\begin{proof}
Since $S_1(V,W)=S_0(V,V)$ in this case, the result follows immediately from Lemmas \ref{varlemma2} (i), \ref{varlemma1}, \ref{s0lemma}, \ref{s0lemma2} and \eqref{ineqrhosigma}.
\end{proof}

\begin{corollary}\label{varcor2}
If, in the situation of Lemma \ref{varlemma2} (ii), we additionally have $p<q$, then
\begin{align*}
 &\sum_{\substack{M\subseteq[n\vee m]:\\ |M|\leq p+q-1}} \Var\bigl(U_M(V,W)\bigr)\leq\bigl(C_{p,q}+q\bigr)\max\bigl(\sigma_{m,V}^2\rho_{n,W}^2,\sigma_{n,W}^2\rho_{m,V}^2\bigr)\\
&\;+\Bigl(\E[V^4]-3\sigma_{m,V}^4+\kappa_q\sigma_{m,V}^2\rho_{m,V}^2\Bigr)^{1/2}\cdot \Bigl(\E[W^4]-\sigma_{n,W}^4\Bigr)^{1/2}\,,
\end{align*}
where $\kappa_q$ is a finite constant depending only on $q$.
\end{corollary}

\begin{proof}
This follows from Lemma \ref{varlemma2} (ii), Lemma \ref{varlemma1} (applied to $V$) and from the obvious facts that 
\begin{align*}
&\sum_{\substack{M\subseteq[m\wedge n]:\\ 1\leq|M|\leq 2(p\vee q)-1}}\Var\bigl(U_M(V)\bigr)\leq \sum_{\substack{M\subseteq[m]:\\ 1\leq|M|\leq 2q-1}}\Var\bigl(U_M(V)\bigr)\,,\\
&\sum_{\substack{M\subseteq[m\wedge n]:\\ 1\leq|M|\leq 2(p\vee q)-1}}\Var\bigl(U_M(W)\bigr)\leq \sum_{\substack{M\subseteq[n]:\\ 1\leq|M|\leq 2p}}\Var\bigl(U_M(W)\bigr) =\Var(W^2)
=\E[W^4]-\sigma_{n,W}^4.
\end{align*}
\end{proof}

\section{Proofs}\label{Proofs}
In this Section we provide detailed proofs of Theorem \ref{maintheo}, Theorem \ref{relcomp} and relation \eqref{e:tricky}. 
\subsection{Proof of Theorem \ref{maintheo}}
The proof of Theorem \ref{maintheo} will follow the classical two-step procedure (see \cite[Section15]{Bil1}) of establishing convergence of finite dimensional distributions and checking tightness. 

\subsubsection{Convergence of finite dimensional distributions}\label{fidi}
We fix time points $0\leq t_1<t_2<\ldots<t_l\leq 1$ and, with $r:=ld$, consider the random vector $V=V^{(m)}=(V_1,\dotsc,V_r)^T\in\R^r$ of degenerate $U$-statistics, defined as follows. 
Given an integer $1\leq i\leq r$, write $i= al+b$ with integers $0\leq a\leq d-1$ and $1\leq b\leq l$ and define 
\[V_i:=\mathbf{W}^{(m)}_{t_b}(a+1)\,.\]  
In other words, we have 
\begin{equation*}
 V=\bigl(\mathbf{W}^{(m)}_{t_1}(1),\dotsc,\mathbf{W}^{(m)}_{t_l}(1),\mathbf{W}^{(m)}_{t_1}(2),\dotsc,\mathbf{W}^{(m)}_{t_l}(2),\dotsc,\mathbf{W}^{(m)}_{t_1}(d),\dotsc,\mathbf{W}^{(m)}_{t_l}(d)\bigr)^T\,.
\end{equation*}
Moreover, for the same $i$ we write $n_{m}(i):=\floor{n_m t_b}$, $\sigma_m(i)^2:=\Var(V_i)$, $q_i:=p_{a+1}$ and 
\[\phi_m(i)^2:=\max_ {1\leq j\leq n_m} \sum_{\substack{J\in \D_{p_{a+1}}(n_m(i)):\\ j\in J}} \E\bigl[W^{(m)}_J(a+1)^2\bigr]\,.\]
Then, we have the straightforward inequalities
\begin{align}\label{fd1}
\sigma_m(i)^2\leq 1\quad\text{and}\quad \phi_m(i)^2\leq\rho_{m,a+1}^2\leq \max_{1\leq k\leq d}\rho_{m,k}^2\,.
\end{align}
Let $\Sigma_m=(s^{(m)}_{i,k})_{1\leq i,k\leq r}\in\R^{r\times r}$ denote the covariance matrix of $V$. Then, for $1\leq i\leq k\leq r$, 
such that 
\begin{equation}\label{repik}
i=al+b,\quad k=a'l+b',\quad \text{where}\quad 0\leq a\leq a'\leq d-1\quad \text{and}\quad 1\leq b,b'\leq l\,,
\end{equation}
 we have 
\[s^{(m)}_{i,k}=\E_m\bigl[\mathbf{W}^{(m)}_{t_b}(a+1) \mathbf{W}^{(m)}_{t_{b'}}(a'+1)\bigr]=\delta_{a,a'}\E_m\bigl[\mathbf{W}^{(m)}_{t_b}(a+1)^2\bigr]=\delta_{a,a'}\sigma_m^2(i)\,.\]
In particular, the diagonal elements of $\Sigma_m$ are given by $s^{(m)}_{i,i}=\sigma_m^2(i)$, $1\leq i\leq r$. Thus, thanks to Condition \ref{cond1} %under the assumptions of Theorem \ref{maintheo} 
we have that $\Sigma_m$ converges to the covariance matrix $\Sigma=(s_{i,k})_{1\leq i,k\leq r}$, whose elements are given by 
\begin{equation}\label{limvar}
s_{i,k}=\delta_{a,a'}\lim_{m\to\infty}\sigma_m^2(i)=\delta_{a,a'}v_{a+1}(t_b)\,,
\end{equation}
if again $i$ and $k$ are as in \eqref{repik}.
Finally, we denote by $N^{(m)}=(N^{(m)}_1,\dotsc,N^{(m)}_r)^T$ and $N=(N_1,\dotsc,N_r)^T$ centered Gaussian vectors on $(\Omega,\F,\P)$ with covariance matrices $\Sigma_m$ and $\Sigma$, respectively.

Next we make use of the well-known fact that the collection $\mathcal{H}_3$ of all functions $h\in C^3(\R^r)$, each of whose partial derivative of order $\leq3$ is uniformly bounded, is convergence determining for weak convergence on $\R^3$. Then, by the triangle inequality we have 
\begin{align}\label{fd2}
&\babs{\E_m[h(V)]-\E[h(N)]}\leq\babs{\E_m[h(V)]-\E[h(N^{(m)})]}+\babs{\E[h(N^{(m)})]-\E[h(N)]}\,.
\end{align}
Note that the second term on the right hand side of \eqref{fd2} converges to zero as $m\to\infty$, since $\lim_{m\to\infty}\Sigma_m=\Sigma$. 

In order to deal with the first term as well, we apply the following central lemma, which is a slightly modified version of \cite[Lemma 4.1 (a)]{DP18}, adapted to the present notation.
The main difference is that in the statement of the result in \cite{DP18}, the term $\phi_m(i)^2$ has already been replaced with the more concrete term
$q_i/n_m(i)$, which is its value in the particular case of symmetric $U$-statistics. The proof given in \cite{DP18}, however,  works in the more general present situation as well.

In what follows, for $1\leq i,k\leq r$, we let
\begin{equation*}
V_iV_k=\sum_{\substack{M\subseteq[n_m(i)\vee n_m(k)]:\\\abs{M}\leq q_i+q_k}} U_M(i,k)
\end{equation*}
denote the Hoeffding decomposition of $V_iV_k$. 

\begin{lemma}\label{mdimlemma}
 Under the above assumptions, there are constants $\kappa_{q_i}\in(0,\infty)$, only depending on $q_i$, $1\leq i\leq r$, such that the following holds: For any $h\in \mathcal{H}_3$, %such that $\E_m\bigl[\abs{h(V)}\bigr]<\infty$ and $\E\bigl[\abs{h(N)}\bigr]<\infty$, 
there are constants $C_1(h),C_2(h)\in[0,\infty)$ such that  
\begin{align}\label{fd3}
&\babs{\E_m[h(V)]-\E[h(N^{(m)})]}\leq \frac{C_1(h)}{4q_1}\sum_{i,k=1}^r(q_i+q_k)\biggl(\sum_{\substack{M\subseteq[n_m(i)\vee n_m(k)]:\\\abs{M}\leq q_i+q_k-1}}\Var_m\bigl(U_M(i,k)\bigr)\biggr)^{1/2}\notag\\
&\;\hspace{4.5cm}+\frac{2C_2(h)\sqrt{r}}{9q_1}\sum_{i=1}^r q_i\sigma_m(i)\biggl(\sum_{\substack{M\subseteq[n_m(i)]:\\\abs{M}\leq 2q_i-1}}\Var_m\bigl(U_M(i,i)\bigr)\biggr)^{1/2}\notag\\
&\;\hspace{4.5cm}+\frac{\sqrt{2r}C_2(h)}{9q_1}\sum_{i=1}^r q_i\sqrt{\kappa_{q_i}}  \sigma_m(i)^3\phi_m(i)\,.%  \rho_n(i)\,.
\end{align}
\end{lemma}

We will also rely on the following crucial lemma, which essentially goes back to de Jong's monograph \cite{deJo89}. 
\begin{lemma}\label{cumlemma}
Under Conditions \ref{lindcond} and \ref{fmcond} %the assumptions of Theorem \ref{maintheo} 
and with the notation of this subsection, we have that 
\begin{equation*}
\lim_{m\to\infty}\E[V_i^4]-3\sigma_m^4(i)=0
\end{equation*}
for all $1\leq i\leq r$.
\end{lemma}

\begin{proof}
Fix $1\leq i\leq r$ given as in \eqref{repik}. Then, we have 
\[V_i=\sum_{J\in\D_{p_{a+1}}(n_m(i))}W_J^{(m)}(a+1)=\sum_{J\in\D_{p_{a+1}}(n_m)}b_J W_J^{(m)}(a+1)\,,\]
where, for $J\in\D_{p_{a+1}}(n_m)$, we let 
\[b_J=\prod_{j\in J} \1_{\{1\leq j\leq n_m(i)\}}\,.\]
This ensures that the coefficient sequence $(b_J)_{J\in\D_{p_{a+1}}(n_m)}$ is \textit{of rank one} in de Jong's terminology (see \cite[Subsection 4.1]{deJo89} and Remark \ref{monorem} (b), below) and we obviously have $0\leq \1_{\{1\leq j\leq n_m(i)\}}\leq1$ for all $1\leq j\leq n_m$. Hence, by \cite[Proposition 4.1.2.]{deJo89} we have $\lim_{m\to\infty}\E[V_i^4]-3\sigma_m^4(i)=0$.
\end{proof}

\begin{remark}\label{monorem}
\begin{enumerate}[(a)]
\item  We mention that the statement of \cite[Proposition 4.1.2.]{deJo89} actually only implies the statement $\lim_{m\to\infty}\E[V_i^4]-3\sigma_m^4(i)=0$ of Lemma \ref{cumlemma} under the stronger condition that $D_{k,m}$ is bounded in $m$ for each $1\leq k\leq d$. However, as can be seen from the proof of \cite[Proposition 4.1.2.]{deJo89}, the result remains true under the more relaxed Condition \ref{lindcond}, since the quantities $\tau'$ and $(\tau')^*$ appearing there still converge to zero, thanks to the inequality 
 \[(\tau')^*\leq D_{m}\tau'\lesssim  D_m \rho_m^2\,,\]
the second part of which holds true by virtue of \cite[Proposition 2.9]{DP17}. The first (and easy) part of it is proved in \cite{deJo89}.
\item In \cite{deJo89} the following definition of a \textit{rank one coefficient sequence} $(a_J)_{J\in\D_p(n_m)}$ is given: One has $|a_J|\leq 1$ for all $J\in\D_p(n_m)$ and there is a sequence $(c_j)_{1\leq j\leq n_m}$ such that $a_J=\prod_{j\in J} c_j$ for all $J\in\D_p(n_m)$. 
\end{enumerate}
\end{remark}
\medskip

In order to prove that the first term on the right hand side of \eqref{fd2} converges to zero, we will make sure that the right hand side of \eqref{fd3} goes to zero as $m\to\infty$. By \eqref{fd1} and the assumption of Theorem \ref{maintheo}, for the last term we have 
\begin{equation}\label{fd6}
 \lim_{m\to\infty}\frac{\sqrt{2r}C_2(h)}{9q_1}\sum_{i=1}^r q_i\sqrt{\kappa_{q_i}}  \sigma_m(i)^3\phi_m(i)=0.
\end{equation}

 In order to deal with the first two terms, we invoke Corollaries \ref{varcor1} and \ref{varcor2}. Thus, suppose again that 
$1\leq i\leq k\leq r$ are given by \eqref{repik}.

\textbf{Case 1:} $a<a'$: In this case, we have $q_i=p_{a+1}<p_{a'+1}=q_k$ and, hence, we are in the situation of Corollary \ref{varcor2}. Thus, we have that 
\begin{align}\label{fd4}
&\sum_{\substack{M\subseteq[n_m(i)\vee n_m(k)]:\\\abs{M}\leq q_i+q_k-1}}\Var_m\bigl(U_M(i,k)\bigr)\leq
\bigl(C_{q_i,q_k}+q_k\bigr)\max\bigl(\sigma^2_m(i)\phi^2_m(k),\sigma^2_m(k)\phi^2_m(i)\bigr)\notag\\
&\;+\Bigl(\E[V_k^4]-3\sigma_m^4(k)+\kappa_{q_k}\sigma_{m}^2(k)\phi_{m}^2(k)\Bigr)^{1/2}\cdot \Bigl(\E[V_i^4]-\sigma_{m}^4(i)\Bigr)^{1/2}\notag\\
&\leq \bigl(C_{q_i,q_k}+q_k\bigr)\max_{1\leq u\leq d}\rho_{m,u}^2\notag\\
&\;+\Bigl(\E[V_k^4]-3\sigma_m^4(k)+\kappa_{q_k}\max_{1\leq u\leq d}\rho_{m,u}^2\Bigr)^{1/2}\cdot \Bigl(\E[V_i^4]\Bigr)^{1/2}\,,
\end{align}
where we have applied \eqref{fd1} for the second inequality and used $\kappa_{q_k}$ to denote a finite constant depending only on $q_k$. 
By Lemma \ref{cumlemma} and Condition \ref{lindcond}, the right hand side of \eqref{fd4} thus converges to zero as $m\to\infty$. 

\textbf{Case 2:} $a=a'$: In this case, we have $q_i=p_{a+1}=p_{a'+1}=q_k$ and, hence, we are in the situation of Corollary \ref{varcor1}. Thus, we have that 
\begin{align}\label{fd5}
&\sum_{\substack{M\subseteq[n_m(i)\vee n_m(k)]:\\\abs{M}\leq q_i+q_k-1}}\Var_m\bigl(U_M(i,k)\bigr)
=\sum_{\substack{M\subseteq[n_m(k)]:\\\abs{M}\leq q_i+q_k-1}}\Var_m\bigl(U_M(i,k)\bigr)\notag\\
&\leq \E[V_i^4]-3\sigma_{m}^4(i)+2q_i\sigma_{m}^2(i)\phi_{m}^2(i) +\bigl(C_{q_i,q_i}+2q_i \bigr)\phi_{m}^2(k)\sigma_{m}^2(k)\notag\\
&\;+\biggl(\E[V_i^4]-3\sigma_{m}^4(i)+\kappa_{q_i}\sigma_{m}^2(i)\phi_{m}^2(i)\biggr)^{1/2}\cdot \biggl(\E[V_k^4]-3\sigma_{m}^4(k)+\kappa_{q_k}\sigma_{m}^2(k)\phi_{m}^2(k)\biggr)^{1/2}\notag\\
&\leq \E[V_i^4]-3\sigma_{m}^4(i)+2q_i \max_{1\leq u\leq d}\rho_{m,u}^2+\bigl(C_{q_i,q_i}+2q_i \bigr)\max_{1\leq u\leq d}\rho_{m,u}^2\notag\\
&\;+\biggl(\E[V_i^4]-3\sigma_{m}^4(i)+\kappa_{q_i}\max_{1\leq u\leq d}\rho_{m,u}^2\biggr)^{1/2}\cdot \biggl(\E[V_k^4]-3\sigma_{m}^4(k)
+\kappa_{q_k}\max_{1\leq u\leq d}\rho_{m,u}^2\biggr)^{1/2}\,,
\end{align}
where we have again applied \eqref{fd1} for the second inequality and used $\kappa_{q_k}$ and $\kappa_{q_i}$ to denote finite constants depending only on $q_k$ and $q_i$, respectively.  Again, by Lemma \ref{cumlemma} and Condition \ref{lindcond}, the right hand side of \eqref{fd5} thus converges to zero as $m\to\infty$. 

From \eqref{fd3},\eqref{fd1}, \eqref{fd6}, \eqref{fd4} and \eqref{fd5} we thus conclude that 
\begin{equation*}
 \lim_{m\to\infty} \babs{\E_m[h(V)]-\E[h(N^{(m)})]}=0
\end{equation*}
and, thus, by \eqref{fd2} that 
\begin{equation*}
 \lim_{m\to\infty} \babs{\E_m[h(V)]-\E[h(N)]}=0
\end{equation*}
for each $h\in\mathcal{H}_3$.
Note further that, by the definition of $\Sigma$, we have the distributional identity 
\begin{align*}
 N\stackrel{\D}{=}\bigl(\mathbf{Z}_{t_1}(1),\dotsc,\mathbf{Z}_{t_l}(1),\dotsc,\mathbf{Z}_{t_1}(d),\dotsc,\mathbf{Z}_{t_l}(d)\bigr)^T\,.
\end{align*}
Hence, we have established the convergence of the finite dimensional distributions of $\mathbf{W}$ to those of $\mathbf{Z}$.  

\subsubsection{Conclusion of the argument}\label{tightness}

We will now finish the proof of Theorem \ref{maintheo} by proving tightness of the sequence $(\mathbf{W}^{(m)})_{m\in\N}$ in $D[0,1]$. This will be done by verifying that 
\begin{equation}\label{tight1}
\omega\bigl(\mathbf{W}^{(m)},\delta\bigr)\stackrel{\P_m}{\longrightarrow}0\,,\quad\text{as }m\to\infty\,,
\end{equation}
where, for $x\in D[0,1]$ and $\delta>0$,
\[\omega(x,\delta):=\sup\{|x(t)-x(s)|\,:\, s,t\in[0,1]\,, |t-s|<\delta\}\]
denotes the \textit{modulus of continuity} of $x$.
Since $\mathbf{W}^{(m)}_0=0$ for all $m\in\N$ and by the finite-dimensional distribution convergence established in Subsection \ref{fidi}, this will imply Theorem \ref{maintheo} by the Corollary to 
\cite[Theorem 13.4]{Bil}. Let us introduce the following notation: Fix $1\leq l\leq d$ that we will from now on suppress from the notation whenever convenient, and, for $m\in\N$ and $0\leq j\leq n_m$, define $\Delta_j^{(m)}:= \mathbf{W}^{(m)}_{\frac{j}{n_m}}(l)-\mathbf{W}^{(m)}_{\frac{j-1}{n_m}}(l)$, $\F_j^{(m)}:=\sigma(X_1^{(m)},\dotsc,X_j^{(m)})$ and $S_j^{(m)}:=\sum_{i=1}^j \Delta_i^{(m)}$. Then, by degeneracy, for each $m\in\N$, $(S_j^{(m)})_{1\leq j\leq n_m}$ is an $( \F_j^{(m)})_{0\leq j\leq n_m}$- martingale. Note that with this notation at hand, we have $\mathbf{W}^{(m)}_t=S_{\floor{n_mt}}^{(m)}$ for each $t\in[0,1]$ and \cite[Theorem 8.4]{Bil1} and the Corollary to \cite[Theorem 13.4]{Bil} together imply that \eqref{tight1} holds, if, for each $k\in\N$ we have 
\begin{equation}\label{tight2}
\lim_{\lambda\to\infty}\limsup_{m\to\infty}\lambda^2\P_m\Bigl(\max_{0\leq i\leq n_m-k}\babs{S_{k+i}^{(m)}-S_k^{(m)}}\geq\lambda\Bigr)=0\,. 
\end{equation}
Note that the additional $\sqrt{n}$ appearing in the statement of \cite[Theorem 8.4]{Bil1} does not appear here, since our processes are already normalized at time $t=1$.
Let us thus fix $k\in\N$. Then, the sequence $(S_{k+i}^{(m)}-S_k^{(m)})_{0\leq i\leq n_m-k}$ is an $( \F_{k+i}^{(m)})_{0\leq i\leq n_m-k}$- martingale. Thus, by first applying Markov's and then Doob's $L^4$-inequality, we have that 
\begin{align}\label{tight3}
&\lambda^2\P_m\Bigl(\max_{0\leq i\leq n_m-k}\babs{S_{k+i}^{(m)}-S_k^{(m)}}\geq\lambda\Bigr)\notag\\
&\leq\E_m\Bigl[\Bigl(\max_{0\leq i\leq n_m-k}\babs{S_{k+i}^{(m)}-S_k^{(m)}}\Bigr)^2\,\1_{\{\max_{0\leq i\leq n_m-k}\abs{S_{k+i}^{(m)}-S_k^{(m)}}\geq\lambda\}}\Bigr]\notag\\
&\leq\frac{1}{\lambda^2}\E_m\Bigl[\Bigl(\max_{0\leq i\leq n_m-k}\babs{S_{k+i}^{(m)}-S_k^{(m)}}\Bigr)^4\Bigr]
\leq \Bigl(\frac{4}{3}\Bigr)^4\frac{1}{\lambda^2}\E_m\Bigl[\babs{S_{n_m}^{(m)}-S_k^{(m)}}^4\Bigr]\notag\\
&\leq \Bigl(\frac{4}{3}\Bigr)^4\frac{8}{\lambda^2}\Bigl(\E_m\babs{S_{n_m}^{(m)}}^4 +\E_m\babs{S_k^{(m)}}^4\Bigr)\,.
\end{align}
Now, firstly, by Condition \ref{fmcond}, we have 
\begin{equation}\label{tight4}
\lim_{m\to\infty}\E_m\babs{S_{n_m}^{(m)}}^4=\lim_{m\to\infty}\E_m\bigl[W^{(m)}(l)^4]=3\,.
\end{equation}
Secondly, using the inequality $|\sum_{i=1}^n a_i|^4\leq n^3 \sum_{i=1}^n |a_i|^4$, and by Condition \ref{lindcond} we obtain that   
\begin{align}\label{tight5}
\E_m\babs{S_k^{(m)}}^4&=\E_m\Babs{\sum_{J\in D_{p_l}(k)} W_J^{(m)}(l)}^4\leq \binom{k}{p}^3 \sum_{J\in D_{p_l}(k)}\E_m\babs{W_J^{(m)}(l)}^4\notag\\
&\leq \binom{k}{p_l}^3 D_{m,l} \sum_{J\in D_{p_l}(k)}\Bigl(\E_m\babs{W_J^{(m)}(l)}^2\Bigr)^2\notag\\
&\leq \binom{k}{p_l}^3 D_{m,l} \rho_{m,l}^2 \sum_{J\in D_{p_l}(k)}\E_m\babs{W_J^{(m)}(l)}^2\notag\\
%\leq \binom{k}{p_l}^3 D_{m,l} \sum_{J\in D_{p_l}(k)}\E_m\babs{W_J^{(m)}(l)}^2\notag\\
%&=\binom{k}{p_l}^3 D_{m,l}\sum_{i=p_l}^k \sum_{\substack{J\in D_{p_l}(k):\\ \max(J)=i}}\E_m\babs{W_J^{(m)}(l)}^2
&\leq \binom{k}{p_l}^3 D_{m,l}\rho_{m,l}^2
\stackrel{m\to\infty}{\longrightarrow}0\,,
\end{align}
where we have used the fact that $\E_m\babs{W_J^{(m)}(l)}^2\leq \rho_{m,l}^2$ for all $J\in D_{p_l}(k)$.
Thus, \eqref{tight3}-\eqref{tight5} imply that 
\[\limsup_{m\to\infty}\lambda^2\P_m\Bigl(\max_{0\leq i\leq n_m-k}\babs{S_{k+i}^{(m)}-S_k^{(m)}}\geq\lambda\Bigr)\leq \Bigl(\frac{4}{3}\Bigr)^4\frac{24}{\lambda^2}=\frac{2048}{27}\,\frac{1}{\lambda^{2}} \]
for each $\lambda>0$ and $k\in\N$ fixed, so that \eqref{tight2} is indeed satisfied. This finishes the argument for tightness. 
\medskip

\subsection{An additional argument explaining the continuity of the limiting process $\mathbf{Z}$}
Note that, by \cite[Theorem 13.4]{Bil} our conditions necessarily entail the remarkable fact that the limiting Gaussian process $\mathbf{Z}$ has continuous paths, which might be quite surprising at first glance. Thus, we will now give another argument for this fact, which additionally explains why it is true. To this end, define the vector $\mathbf{V}^{(m)}=(\mathbf{V}^{(m)}(1),\dotsc,\mathbf{V}^{(m)}(d))$ of continuous processes, defined by linear interpolation, i.e. for $1\leq l\leq d$ we let 
\begin{align*}
\mathbf{V}^{(m)}_t(l)&:=\mathbf{W}^{(m)}_{\frac{i-1}{n_m}}(l)+ \bigl(n_mt -(i-1)\bigr)\bigl(\mathbf{W}^{(m)}_{\frac{i}{n_m}}(l)-\mathbf{W}^{(m)}_{\frac{i-1}{n_m}}(l)\bigr)\,,\quad \frac{i-1}{n_m}\leq t<\frac{i}{n_m},
\end{align*} 
and $\mathbf{V}^{(m)}_1(l)= \mathbf{W}^{(m)}_1(l)=W^{(m)}(l)$. Then, for fixed $1\leq l\leq d$ and with the same notation as above, it is easy to see that 
\[\norm{\mathbf{V}^{(m)}(l)-\mathbf{W}^{(m)}(l)}_\infty=\max_{1\leq i\leq n_m}\babs{\Delta_i^{(m)}}\,.\]
Thus, using e.g. \cite[Theorem 3]{ISH} in the third inequality below, for $\epsilon>0$ we can bound
\begin{align*}
&\P_m\bigl(\norm{\mathbf{V}^{(m)}(l)-\mathbf{W}^{(m)}(l)}_\infty>\epsilon\bigr)=\P_m\bigl(\max_{1\leq i\leq n_m}\babs{\Delta_i^{(m)}}>\epsilon\bigr)
\leq\sum_{i=1}^{n_m}\P_m\bigl(\babs{\Delta_i^{(m)}}>\epsilon\bigr)\\
&\leq \epsilon^{-4}\sum_{i=1}^{n_m}\E_m\Babs{\sum_{\substack{J\in\D_{p_l}(i):\\ i\in J}} W^{(m)}_J(l)}^4
\leq B(p_l)\epsilon^{-4}\sum_{i=1}^{n_m}\E_m\biggl[\biggl(\sum_{\substack{J\in\D_{p_l}(i):\\ i\in J}} W^{(m)}_J(l)^2\biggr)^2\biggr]\\
&=B(p_l)\epsilon^{-4}\sum_{i=1}^{n_m}\sum_{\substack{J,K\in\D_{p_l}(i):\\ i\in J\cap K}}\E_m\bigl[W^{(m)}_J(l)^2 W^{(m)}_K(l)^2\bigr]\\
&\leq B(p_l)\epsilon^{-4}\sum_{i=1}^{n_m}\sum_{\substack{J,K\in\D_{p_l}(i):\\ i\in J\cap K}}\Bigl(\E_m\bigl[W^{(m)}_J(l)^4\bigr]\E_m\bigl[W^{(m)}_K(l)^4\bigr] \Bigr)^{1/2}\\
&\leq B(p_l)\epsilon^{-4}D_{m,l}\sum_{i=1}^{n_m}\sum_{\substack{J,K\in\D_{p_l}(i):\\ i\in J\cap K}}\E_m\bigl[W^{(m)}_J(l)^2\bigr]\E_m\bigl[W^{(m)}_K(l)^2\bigr] \\
&=B(p_l)\epsilon^{-4}D_{m,l}\sum_{i=1}^{n_m}\biggl(\sum_{\substack{J\in\D_{p_l}(i):\\ i\in J}}\E_m\bigl[W^{(m)}_J(l)^2\bigr]\biggr)^2\\
&\leq B(p_l) \epsilon^{-4}D_{m,l}\rho_{m,l}^2\sum_{i=1}^{n_m}\sum_{\substack{J\in\D_{p_l}(i):\\ i\in J}}\E_m\bigl[W^{(m)}_J(l)^2\bigr]\\
&=B(p_l)\epsilon^{-4}D_{m,l}\rho_{m,l}^2\stackrel{m\to\infty}{\longrightarrow}0\,,
\end{align*} 
where $B(p_l)$ is a finite constant only depending on $p_l$, see \cite[Theorem 3]{ISH}.
Thus, we have proved that, for each $1\leq l\leq d$, $\norm{\mathbf{V}^{(m)}(l)-\mathbf{W}^{(m)}(l)}_\infty$ and, a fortiori their Skorohod distance, converges to zero in probability as $m\to\infty$, so that \cite[Theorem 3.1]{Bil} implies that the processes $\mathbf{V}^{(m)}(l)$ and $\mathbf{W}^{(m)}(l)$ must indeed have the same distributional limit, if any.

\subsection{Proof of Theorem \ref{relcomp}}\label{ss:proofrelcomp}
Suppose that Conditions \ref{lindcond} and \ref{fmcond} hold for the sequence $(\mathbf{W}^{(m)})_{m\in\N}$. Since the above argument for tightness does not make use of Condition \ref{cond1}, tightness still holds. Let $(\mathbf{W}^{(m_l)})_{l\in\N}$ be a given subsequence. By relative compactness, there is a further subsequence $(\mathbf{W}^{(m_{l_k})})_{k\in\N}$ that converges in distribution to some process $\mathbf{Z}$. Again by \cite[Theorem 13.4]{Bil} we can conclude that $\mathbf{Z}$ has values in $C([0,1];\R^d)$. Thus, the projections $\pi_t: D([0,1];\R^d)\rightarrow\R^d$, $x\mapsto x(t)$, $t\in[0,1]$, are $\P_{\mathbf{Z}}$-a.s. continuous, where $\P_{\mathbf{Z}}$ denotes the law of $\mathbf{Z}$ on $D([0,1];\R^d)$. Hence, the continuous mapping theorem implies that the finite-dimensional distributions of $(\mathbf{W}^{(m_{l_k})})_{k\in\N}$ converge to those of $\mathbf{Z}$ and, e.g. by a uniform integrability argument, one makes sure that the covariance matrices of the finite dimensional distributions of $(\mathbf{W}^{(m_{l_k})})_{k\in\N}$ also converge to the respective ones of $\mathbf{Z}$. Thus, the argument leading to \eqref{limvar} goes through with the right hand side replaced by the variance of the corresponding coordinate of $\mathbf{Z}$ at $t_b$. Now, since Lemma \ref{cumlemma} applies under Conditions \ref{lindcond} and \ref{fmcond}, the same reasoning as in Subsubsection \ref{fidi} ensures that the finite dimensional distributions of $(\mathbf{W}^{(m_{l_k})})_{k\in\N}$ have a Gaussian limit, implying that $\mathbf{Z}$ is indeed a continuous, centered Gaussian process. {\red The non-correlation, and therefore the independence, of the components of ${\bf Z}$ follows from the fact that Hoeffding degenerate $U$-statistics of distinct orders are orthogonal in $L^2$. The proof is complete.}
{{}
\subsection{Proof of relation \eqref{e:tricky}}\label{ss:tricky}

We fix $p\geq 3$ and $2\leq a\leq p-1$ as in Example \ref{ex:fractional}, and adopt the notation introduced therein. It is sufficient to show that
\begin{equation}\label{e:tricky2}
|F^0_m \cap [ \lfloor tm \rfloor ] ^p | \sim |F^0_{\lfloor tm\rfloor}|,
\end{equation}
a relation that is implied by the next statement.

\begin{lemma} Consider a nondecreasing integer-valued sequence $(\ell(m))_{m\geq 1}$ such that, as $m\to \infty$, $\ell(m) \sim \alpha\cdot m$, for some $\alpha\in (0,1)$. Then
$$
|F^0_m \cap [ \ell(m)]^p | \sim |F^0_{ \ell(m) }|.
$$

\end{lemma}
\begin{proof} For every $\ell\geq 1$ we write $k = k(\ell)$ to denote the unique integer such that $\ell\in [k^a, (k+1)^a)$, in such a way that $F_{\ell} = F_{k^a(\ell)}$. We have that 
$$
|F^0_m \cap [ \ell(m)]^p | = |F^0_{ \ell(m) }| + |  (F^0_m \cap [\ell(m)]^p ) \backslash F^0_{ k^a(\ell(m)) }  |.
$$
Using the relations
$$
|  (F^0_m \cap [\ell(m)]^p ) \backslash F^0_{ k^a(\ell(m)) }  | \leq | F^0_{ (k(\ell(m))+1)^a } | - | F^0_{ k^a(\ell(m)) } | \sim  k(\ell(m))^{p-1} \leq \ell(m)^{(p-1)/a} = o(m^{p/a}),
$$
we immediately deduce the desired conclusion.

\end{proof}

}
\normalem
\bibliographystyle{alpha}
\bibliography{uprocess}
\end{document}